\documentclass[11pt,a4paper]{amsart}
\usepackage[all]{xy}
\SelectTips{cm}{}
\calclayout

\usepackage{amscd,amssymb,amsopn,amsmath,amsthm,graphics,mathrsfs,accents}
\usepackage[colorlinks=true,linkcolor=blue,citecolor=blue]{hyperref}

\newtheorem{theorem}{Theorem}[section]
\newtheorem{proposition}[theorem]{Proposition}
\newtheorem{lemma}[theorem]{Lemma}
\newtheorem{corollary}[theorem]{Corollary}
 
\newtheorem*{itheorem}{Theorem}
\newtheorem*{iproposition}{Proposition}
\newtheorem*{icorollary}{Corollary}

\theoremstyle{definition}
\newtheorem{definition}[theorem]{Definition}
\newtheorem{example}[theorem]{Example}
\newtheorem{remark}[theorem]{Remark}
\newtheorem{construction}[theorem]{Construction}
\newtheorem{chunk}[subsection]{}

\newtheorem*{setup}{Setup}

\DeclareMathOperator{\Map}{Map}
\DeclareMathOperator{\Flat}{Flat}
\DeclareMathOperator{\Inj}{Inj}
\DeclareMathOperator{\inc}{inc}
\DeclareMathOperator{\Proj}{Proj}
\DeclareMathOperator{\can}{can}
\DeclareMathOperator{\Ext}{Ext}
\DeclareMathOperator{\ac}{ac}
\DeclareMathOperator{\tac}{tac}
\DeclareMathOperator{\qc}{qc}
\DeclareMathOperator{\K}{\mathbf{K}}
\DeclareMathOperator{\coh}{coh}
\DeclareMathOperator{\pac}{pac}
\DeclareMathOperator{\Hom}{Hom}
\DeclareMathOperator{\modd}{mod}

\DeclareMathOperator{\Thick}{\mathbf{Thick}}
\DeclareMathOperator{\Perf}{\mathbf{P}}
\DeclareMathOperator{\Coperf}{\mathbf{I}}

\DeclareMathOperator{\ssg}{Ssg}
\DeclareMathOperator{\gsg}{Gsg}
\DeclareMathOperator{\Qco}{Qco}
\DeclareMathOperator{\Ker}{Ker}
\DeclareMathOperator{\Spec}{Spec}
\DeclareMathOperator{\Loc}{Loc}

\begin{document}
\newcommand{\kproj}[1]{\K(\Proj #1)}
\newcommand{\kprojl}[2]{\K_{#1}(\Proj #2)}
\newcommand{\kproju}[2]{\K^{#1}(\Proj #2)}
\newcommand{\kprojul}[3]{\K^{#1}_{#2}(\Proj #3)}
\newcommand{\kinj}[1]{\K(\Inj #1)}
\newcommand{\kinju}[2]{\K^{#1}(\Inj #2)}
\newcommand{\kinjl}[2]{\K_{#1}(\Inj #2)}
\newcommand{\kprof}[1]{\mathbf{N}(\Flat #1)}
\newcommand{\kprofu}[2]{\mathbf{N}^{#1}(\Flat #2)}
\newcommand{\kprofl}[2]{\mathbf{N}_{#1}(\Flat #2)}
\newcommand{\kproful}[3]{\mathbf{N}^{#1}_{#2}(\Flat #3)}
\newcommand{\kflat}[1]{\K(\Flat #1)}
\newcommand{\kflatl}[2]{\K_{#1}(\Flat #2)}
\newcommand{\cat}[1]{\mathcal{#1}}
\newcommand{\qder}[1]{\mathbf{D}(#1)}
\newcommand{\qderl}[2]{\mathbf{D}_{#1}(#2)}
\newcommand{\qderu}[2]{\mathbf{D}^{#1}(#2)}
\newcommand{\qderul}[3]{\mathbf{D}^{#1}_{#2}(#3)}
\newcommand{\md}[1]{\mathscr{#1}}
\newcommand{\xlto}[1]{\stackrel{#1}\lto}
\newcommand{\mf}[1]{\mathfrak{#1}}
\newcommand{\lto}{\longrightarrow}
\newcommand{\qcmod}[1]{\Qco(#1)}
\newcommand{\qcmodo}[1]{\Qco#1}
\newcommand{\cmod}[1]{\coh(#1)}
\newcommand{\cmodn}[1]{\coh#1}
\newcommand{\ppe}[1]{\kflatl{\pac}{#1}}

\def\l{\,|\,}
\def\shom{\mathscr{H}\! om}
\def\qhom{\mathscr{H}\! om_{\qc}}
\def\ab{\bold{Ab}}
\def\holim{\underrightarrow{holim}}
\def\rdev{\mathbb{R}}
\def\ndev{\mathbb{N}}
\def\SSg{\mathbf{D}_{\ssg}^b}
\def\NSg{\mathbf{D}_{\gsg}^b}

\title{Totally Acyclic Complexes over Noetherian Schemes}

\author{Daniel Murfet}
\address{Hausdorff Center for Mathematics, University of Bonn}
\email{murfet@math.uni-bonn.de}

\author{Shokrollah Salarian}
\address{Department of Mathematics, University of Isfahan, P.O.Box:
       81746-73441, Isfahan, Iran}
\email{Salarian@ipm.ir}
\thanks{This research was in part supported by a grant from the University of Isfahan (No. 861125). The second author thanks the Center of Excellence for Mathematics (University of Isfahan).}

\begin{abstract}
We define a notion of total acyclicity for complexes of flat quasi-coherent sheaves over a semi-separated noetherian scheme, generalising complete flat resolutions over a ring. By studying these complexes as objects of the pure derived category of flat sheaves we extend several results about totally acyclic complexes of projective modules to schemes; for example, we prove that a scheme is Gorenstein if and only if every acyclic complex of flat sheaves is totally acyclic. Our formalism also removes the need for a dualising complex in several known results for rings, including J\o rgensen's proof of the existence of Gorenstein projective precovers.
\end{abstract}

\maketitle

\tableofcontents

\section{Introduction}

The derived category of a ring $A$ is the result of formally inverting quasi-isomorphisms in the category of cochain complexes of $A$-modules and cochain maps. The importance of this construction arises from the fact that the standard homological resolutions (projective, injective, flat, etc.) of an $A$-module are isomorphic objects in the derived category, making this a natural setting for homological algebra.

In commutative algebra and the singularity theory of algebraic varieties, there arise modules with a nonstandard resolution known as a ``complete resolution'', and one would like to understand these resolutions in a categorical context akin to the derived category. To be precise, a \emph{complete projective resolution} of a finitely generated module $M$ over a commutative noetherian ring $A$ consists of the following data:
\begin{itemize}
\item An acyclic complex of projective $A$-modules $P$ such that $\Hom_A(P,Q)$ is acyclic for every projective $A$-module $Q$. Such complexes $P$ are said to be \emph{totally acyclic}.
\item An isomorphism $S \cong \Ker(\partial: P^0 \lto P^1)$, where $S$ is some syzygy of $M$.
\end{itemize}
Those modules with complete projective resolutions are said to have \emph{finite G-dimension}, and the study of such modules is known as \emph{Gorenstein homological algebra}; see \cite{AuslanderBridger69,Christensen00,Enochs06}. Many modules admit complete projective resolutions, including all finitely generated modules over Gorenstein rings. Just as the derived category is the natural setting for homological algebra, the homotopy category $\kprojl{\tac}{A}$ of totally acyclic complexes of projective $A$-modules is the correct framework for the homological algebra of complete resolutions. This triangulated category has been studied by several authors, including Buchweitz \cite{Buchweitz}, Beligiannis \cite{Belig00}, J\o rgensen \cite{Jorgensen01, Jorgensen07}, Iyengar and Krause \cite{Krause06} and Xiao-Wu Chen \cite{Chen07}. There is one feature that we need to mention here: the complete projective resolution of an $A$-module $M$ is unique up to isomorphism in this category (that is, the complete resolution is unique up to homotopy equivalence). 

The purpose of this paper is to introduce a triangulated category of totally acyclic complexes of flat sheaves, which plays the role of $\kprojl{\tac}{A}$ for any noetherian scheme. In subsequent work we will show how one can use this formalism to generalise aspects of Gorenstein homological algebra to schemes. The main obstacle in developing the theory is the lack of projective quasi-coherent sheaves over non-affine schemes, which we solve using a recent construction of Neeman \cite{Neeman08}, developed in the PhD thesis \cite{Murfet07} of the first author. The key insight is that one should construct the global theory using complexes of flat, rather than projective, quasi-coherent sheaves.

\begin{setup} In this article $X$ denotes a semi-separated noetherian scheme. The definition of semi-separatedness is recalled in Section \ref{section:notation}. Unless otherwise specified ``sheaf'' means a \textbf{quasi-coherent} sheaf of modules over $X$. All rings are commutative.
\end{setup}

We define a complex of flat sheaves $F$ to be \emph{N-totally acyclic} if it is acyclic and $F \otimes \md{I}$ is acyclic for every injective sheaf $\md{I}$. This generalises a notion introduced for complexes of flat modules by Enochs, Jenda and Torrecillas \cite{Enochs93}. As we will see shortly, these complexes are totally acyclic in a natural sense, once the correct definition of morphisms between complexes of flat sheaves is chosen: the ``N'' in the nomenclature relates to this choice. Next, in order to make this class of N-totally acyclic complexes into a category, we have to specify the morphism sets.

If we take homotopy-equivalence classes of cochain maps as morphisms, we obtain the full subcategory $\kflatl{\tac}{X}$ of N-totally acyclic complexes in the homotopy category $\kflat{X}$ of flat sheaves. This category, however, is the wrong one for our purposes. In our category of N-totally acyclic complexes, we want the generalisation of the complete resolutions described above to be unique (as in, two complete resolutions of the same sheaf are isomorphic) and this fails to hold in $\kflatl{\tac}{X}$.

To force the standard resolutions of homological algebra to be unique, one inverts the quasi-isomorphisms to form the derived category. Something similar applies here; the wrinkle is that we only invert those quasi-isomorphisms $f: F \lto F'$ with the property that $f \otimes \md{A}: F \otimes \md{A} \lto F' \otimes \md{A}$ is a quasi-isomorphism for every sheaf $\md{A}$. Morphisms with this property are called \emph{pure quasi-isomorphisms}. A complex of sheaves $F$ is said to be \emph{pure acyclic} if it is acyclic, and $F \otimes \cat{A}$ is acyclic for every sheaf $\md{A}$. The result of formally inverting all the pure quasi-isomorphisms in $\kflat{X}$ is known as the \emph{pure derived category of flat sheaves}. Formally, this category is the Verdier quotient
\[
\kprof{X} := \kflat{X}/\kflatl{\pac}{X},
\]
where $\kflatl{\pac}{X}$ denotes the full subcategory of pure acyclic complexes in $\kflat{X}$. This construction was developed in \cite{Neeman08, NeemanAdj, Murfet07} and we recall the relevant details in Section \ref{section:review_of_pure}. In this paper, the main category of interest is the full triangulated subcategory $\kprofl{\tac}{X}$ of N-totally acyclic complexes in $\kprof{X}$. Alternatively, since there is an obvious inclusion $\kflatl{\pac}{X} \subseteq \kflatl{\tac}{X}$, we can write 
\[
\kprofl{\tac}{X} = \kflatl{\tac}{X}/\kflatl{\pac}{X}.
\]
The objects of this triangulated category are N-totally acyclic complexes of flat sheaves, and the morphisms are homotopy equivalence classes of cochain maps with pure quasi-isomorphisms formally inverted. This is the ``correct'' triangulated category of N-totally acyclic complexes, as the following theorems attest. 

In what follows we use $[-,-]$ to denote morphism sets in $\kprof{X}$. Our first theorem uses these morphism sets to express N-total acyclicity for complexes of flat sheaves in a form closer to the definitions of total acyclicity for complexes of projective and injective modules over a ring. Recall that a sheaf $\md{C}$ is said to be \emph{cotorsion} if $\Ext^1(\md{P}, \md{C})$ vanishes for every flat sheaf $\md{P}$ (all $\Ext$s are in the category of quasi-coherent sheaves). The class of cotorsion flat sheaves will play a crucial role in this article, and their basic properties are discussed in Section \ref{section:cotorsion}. In our triangulated categories, $\Sigma$ denotes the suspension.

\begin{itheorem} For a complex $F$ of flat sheaves the following are equivalent:
\begin{itemize}
\item[(i)] $F$ is acyclic, and $F \otimes \md{I}$ is acyclic for every injective sheaf $\md{I}$.
\item[(ii)] $F$ is acyclic, and $\Hom(F, \md{C})$ is acyclic for every cotorsion flat sheaf $\md{C}$.
\item[(iii)] $F$ is left and right orthogonal in $\kprof{X}$ to shifts of flat sheaves. That is, for any flat sheaf $\md{P}$ and $i \in \mathbb{Z}$ we have
\[
0 = [F, \Sigma^i \md{P}] = [\Sigma^i \md{P}, F].
\]
\item[(iv)] $F$ is acyclic, and $\Map(F, \md{P})$ is acyclic for every flat sheaf $\md{P}$, where $\Map(-,-)$ denotes the internal Hom in $\kprof{X}$.
\end{itemize}
\end{itheorem}
\emph{Proof.} See Theorem \ref{theorem:description_tac}.
\vspace{0.3cm}

The category $\kprofl{\tac}{X}$ is locally modelled on the homotopy category $\kprojl{\tac}{A}$, in the following sense:  Neeman proves in \cite{Neeman08} that when $X = \Spec(A)$ is affine, the composite of canonical functors
\begin{equation}\label{eq:local_kprof2}
\kproj{A} \xlto{\inc} \kflat{A} \xlto{\can} \kflat{A}/\kflatl{\pac}{A} = \kprof{X}
\end{equation}
is an equivalence, where $\kproj{A}$ is the homotopy category of projective $A$-modules. We prove in Section \ref{subsection:two_kinds} that if $A$ has finite Krull dimension then a complex of projective $A$-modules is totally acyclic if and only if it is N-totally acyclic, and it follows that (\ref{eq:local_kprof2}) restricts to an equivalence
\begin{equation}\label{eq:local_kprof}
\kprojl{\tac}{A} \xlto{\sim} \kprofl{\tac}{X}.
\end{equation}
In general, we view the globally defined category $\kprofl{\tac}{X}$ as being glued together from the homotopy categories of totally acyclic complexes of projective modules over the affine open subsets of $X$. This leads to a general principle: if a statement of Gorenstein homological algebra can be expressed intrinsically in $\kprojl{\tac}{A}$, then it is reasonable to expect a generalisation to schemes using the above framework. As examples of the technique, we generalise results of J\o rgensen \cite{Jorgensen07} and Iyengar-Krause \cite{Krause06}.\\

Let $A$ denote a noetherian ring with a dualising complex. J\o rgensen proves in \cite{Jorgensen07} that the inclusion $\kprojl{\tac}{A} \lto \kproj{A}$ has a right adjoint, and using this adjoint he proves that Gorenstein projective precovers exist over $A$. For a scheme, the category $\kprof{X}$ plays the role of $\kproj{A}$, and an argument analogous to J\o rgensen's yields:

\begin{itheorem} The inclusion $\kprofl{\tac}{X} \lto \kprof{X}$ has a right adjoint. 
\end{itheorem}
\emph{Proof.} See Theorem \ref{theorem:2}.
\vspace{0.3cm}

Given any complex $F$ of flat sheaves, the theorem provides a triangle
\[
T \lto F \lto S \lto \Sigma T
\]
in $\kprof{X}$, where $T$ is N-totally acyclic and $S$ is right orthogonal to N-totally acyclic complexes. Such a triangle has several applications; in particular, it allows one to develop Tate cohomology for noetherian schemes following Krause \cite{Krause05} and J\o rgensen \cite{Jorgensen07}, and it can also be used to develop a theory of Gorenstein dimension for arbitrary sheaves. These topics will be developed in sequels to the present paper.

Iyengar and Krause prove in \cite{Krause06} that a noetherian ring $A$ with a dualising complex is Gorenstein if and only if every acyclic complex of projective $A$-modules is totally acyclic. The generalisation to noetherian schemes reads as follows:

\begin{itheorem} The scheme $X$ is Gorenstein if and only if every acyclic complex of flat sheaves is N-totally acyclic.
\end{itheorem}
\emph{Proof.} See Theorem \ref{theorem:gorenstein_iff_actac}.
\vspace{0.3cm}

The reader will note that these two theorems make no assumption about the existence of a dualising complex, so when $X = \Spec(A)$ is affine and of finite Krull dimension, the results improve the aforementioned theorems of J\o rgensen and Iyengar-Krause.

\begin{itheorem} Any module over a noetherian ring of finite Krull dimension admits a Gorenstein projective precover.
\end{itheorem}
\emph{Proof.} See Theorem \ref{theorem:jorgensen}.
\vspace{0.2cm}

\begin{icorollary} A noetherian ring $A$ of finite Krull dimension is Gorenstein if and only if every acyclic complex of projective $A$-modules is totally acyclic.
\end{icorollary}
\emph{Proof.} See Corollary \ref{corollary:gorenstein_tac_ac_local}.
\vspace{0.3cm}

Let $\kprofl{\ac}{X}$ denote the full subcategory of acyclic complexes in $\kprof{X}$. This is a triangulated subcategory, and both $\kprofl{\tac}{X}$ and $\kprofl{\ac}{X}/\kprofl{\tac}{X}$ are triangulated categories with arbitrary small coproducts, so it is natural to ask whether these categories are compactly generated.

Let $\qderu{b}{\cmodn{X}}$ denote the bounded derived category of coherent sheaves. A complex of sheaves has \emph{finite flat dimension} (resp. \emph{finite injective dimension}) if it is isomorphic, in the derived category $\qder{X}$ of all sheaves, to a bounded complex of flat sheaves (resp. a bounded complex of injective sheaves). We define full subcategories of $\qderu{b}{\cmodn{X}}$ by
\begin{equation}\label{eq:finite_flat_inj_dim}
\begin{split}
\Perf(X) &:= \{ G \in \qderu{b}{\cmodn{X}} \l G \text{ has finite flat dimension} \},\\
\Coperf(X) &:= \{ G \in \qderu{b}{\cmodn{X}} \l G \text{ has finite injective dimension} \}.
\end{split}
\end{equation}
The objects of $\Perf(X)$ are known in the literature as \emph{perfect complexes} (Lemma \ref{lemma:perfx_are_perfect}). We denote by $\Thick(\Perf(X), \Coperf(X))$ the smallest thick subcategory of $\qderu{b}{\cmodn{X}}$ containing the union $\Perf(X) \cup \Coperf(X)$, and define
\begin{equation}
\begin{split}
\SSg(X) &:= \qderu{b}{\cmodn{X}}/\Thick(\Perf(X), \Coperf(X)),\\
\NSg(X) &:= \Thick(\Perf(X), \Coperf(X))/\Perf(X).
\end{split}
\end{equation}
In the next proposition we assume that $X$ has a dualising complex. This is not a strong restriction: any scheme of finite type over a Gorenstein scheme (and in particular any variety over a field) admits a dualising complex; see \cite{ResiduesDuality,Conrad00}. Given a triangulated category $\cat{T}$, we write $\cat{T}^c$ for the full subcategory of compact objects in $\cat{T}$. Recall that a functor $F: \cat{A} \lto \cat{B}$ is an \emph{equivalence up to direct summands} if it is fully faithful, and every object $B \in \cat{B}$ is isomorphic to a direct summand of $F(A)$ for some $A \in \cat{A}$.

\vspace{0.2cm}

The following gives a generalisation to schemes of \cite[Theorem 5.4]{Krause06}.

\begin{iproposition} Suppose that $X$ has a dualising complex. Then both $\kprofl{\tac}{X}$ and the quotient $\kprofl{\ac}{X}/\kprofl{\tac}{X}$ are compactly generated, and there are equivalences up to direct summands
\begin{gather*}
\SSg(X) \lto \kproful{c}{\tac}{X},\\
\NSg(X) \lto \left( \kprofl{\ac}{X}/\kprofl{\tac}{X} \right)^c.
\end{gather*}
\end{iproposition}
\emph{Proof.} See Proposition \ref{prop:compacts_ktac} and Proposition \ref{prop:compacts_dgac}.
\vspace{0.3cm}

Finally, we interpret $\SSg(X)$ and $\NSg(X)$ as invariants of certain singularities, in the same sense that $\qderu{b}{\coh X}/\Perf(X)$ is an invariant of all singularities \cite{Orlov04}. We call $x \in X$ a \emph{symmetric singularity} if $\cat{O}_{X,x}$ is non-regular, and there exists a non-contractible totally acyclic complex of projective $\cat{O}_{X,x}$-modules (equivalently, there exists a non-free Gorenstein projective $\cat{O}_{X,x}$-module). Every Gorenstein singularity is a symmetric singularity, but not every symmetric singularity is Gorenstein (e.g.\ Example \ref{example:non_gorenstein_symmetric}).

\begin{iproposition} Let $U \subseteq X$ be an open subset containing every symmetric singularity of $X$. Then the restriction functor defines an equivalence
\[
(-)|_U: \kprofl{\tac}{X} \xlto{\sim} \kprofl{\tac}{U}.
\]
When $X$ has a dualising complex this induces an equivalence $\SSg(X) \cong \SSg(U)$.
\end{iproposition}
\emph{Proof.} See Proposition \ref{prop:invariance_ntac}.
\vspace{0.1cm}

\begin{iproposition} Let $U \subseteq X$ be an open subset containing every singularity of $X$ that is not Gorenstein. Then the restriction functor defines an equivalence
\[
(-)|_U: \kprofl{\ac}{X}/\kprofl{\tac}{X} \xlto{\sim} \kprofl{\ac}{U}/\kprofl{\tac}{U}.
\]
When $X$ has a dualising complex this induces an equivalence up to direct summands
\[
\NSg(X) \lto \NSg(U).
\]
\end{iproposition}
\emph{Proof.} See Proposition \ref{prop:invariance_nsg}.

\section{Background and notation}\label{section:notation}

\begin{chunk}\label{chunk:schemes}\textbf{Schemes.} A scheme $X$ is said to be \emph{semi-separated} if there exists an open covering $\{ V_\alpha \}_{\alpha \in \Lambda}$ of $X$ such that the $V_\alpha$, and all the pairwise intersections $V_\alpha \cap V_\beta$, are affine; see \cite{ThomasonTrobaugh} and \cite{Tarrio08}. Any separated scheme is semi-separated. Throughout this article $X$ denotes a semi-separated noetherian scheme, all sheaves are quasi-coherent, and $\qcmod{X}$ denotes the category of (quasi-coherent) sheaves on $X$. An unadorned $\Hom(-,-)$ means morphism sets in this category. For any sheaf $\md{F}$ consider the colimit preserving functor $- \otimes \md{F}: \qcmod{X} \lto \qcmod{X}$. This functor has a right adjoint (by the adjoint functor theorem) which we denote by
\[
\qhom(\md{F},-): \qcmod{X} \lto \qcmod{X}.
\]
This defines the internal function object for the closed monoidal structure on $\qcmod{X}$. When $\md{F}, \md{G}$ are sheaves with $\md{F}$ coherent $\shom(\md{F}, \md{G})$ is quasi-coherent, and there is a canonical isomorphism $\qhom(\md{F}, \md{G}) \cong \shom(\md{F}, \md{G})$. 

Since $X$ is noetherian the abelian category $\qcmod{X}$ is locally noetherian and an injective object in $\qcmod{X}$ is just a quasi-coherent sheaf injective in the larger category of all sheaves of $\cat{O}_X$-modules; see \cite[II \S 7]{ResiduesDuality}. It follows that Ext groups are the same in both categories.
\end{chunk}

\begin{chunk}\label{section:background_triangulated_categories} \textbf{Triangulated categories.} The basic references are \cite{NeemanBook,Verdier96}. For material on well-generated triangulated categories, see \cite{NeemanBook,Krause01,Krause08}. Our terminology for localisation sequences and recollements follows \cite{Krause05}, see also \cite{BBD,Verdier96} and \cite[\S 9.2]{NeemanBook}. All functors between triangulated categories are triangulated functors, and all natural transformations of such functors are understood to be compatible with the triangulated structure. A \emph{triangulated category with coproducts} is a triangulated category in which arbitrary small coproducts exist.

A triangulated subcategory $\cat{S}$ of a triangulated category $\cat{T}$ is always assumed to contain any object of $\cat{T}$ isomorphic to an object of $\cat{S}$. A \emph{thick} subcategory is a triangulated subcategory closed under retracts, and a \emph{localising} (resp. \emph{colocalising}) subcategory is a triangulated subcategory closed under arbitrary small coproducts (resp. products). Any localising subcategory of a triangulated category with countable coproducts is automatically thick, by \cite[Proposition 1.6.8]{NeemanBook}. Given a triangulated category $\cat{T}$ with coproducts, an object $C \in \cat{T}$ is \emph{compact} if for every set-indexed family $\{ X_i \}_{i \in I}$ in $\cat{T}$, the canonical map $\oplus_{i \in I} \Hom_{\cat{T}}(C, X_i) \lto \Hom_{\cat{T}}(C, \oplus_{i \in I} X_i)$ is an isomorphism. The compact objects form a thick subcategory $\cat{T}^c$ of $\cat{T}$. The category $\cat{T}$ is \emph{compactly generated} if there is a set $\cat{S}$ of compact objects such that for any $X \in \cat{T}$, if $\Hom_{\cat{T}}(\Sigma^n S, X) = 0$ for all $S \in \cat{S}$ and $n \in \mathbb{Z}$, then $X = 0$.

Given an additive category $\cat{A}$ we write $\K(\cat{A})$ for the homotopy category of $\cat{A}$, which is the triangulated category whose objects are cochain complexes in $\cat{A}$ and whose morphisms are homotopy equivalence classes of cochain maps. Applying this construction to $\qcmod{X}$ and to the full subcategories $\Flat(X)$ and $\Inj(X)$ of flat and injective sheaves, respectively, defines triangulated categories $\K(X) = \K(\qcmodo{X})$, $\kflat{X}$ and $\kinj{X}$. Similarly, given a ring $A$ one defines $\K(A), \kflat{A}, \kinj{A}$ and $\kproj{A}$.

We write $\qder{X}$ for the (unbounded) derived category of $\qcmod{X}$. This category has small Homs and is compactly generated; see \cite{Neeman96, Tarrio00, Tarrio08}. One defines the tensor product $- \otimes^{\mathbb{L}} -$ and internal function object $\rdev\qhom(-,-)$ in $\qder{X}$ using the unbounded $\K$-flat and $\K$-injective resolutions of Spaltenstein \cite{Spalt88}, see \cite{Tarrio08,LipmanNotes} and \cite[\S 6.1]{Murfet07}. This makes $\qder{X}$ into a closed symmetric monoidal category.

Since $X$ is noetherian, the full subcategory of $\qder{X}$ consisting of complexes with bounded coherent cohomology is equivalent to the bounded derived category $\qderu{b}{\cmodn{X}}$ of coherent sheaves, and we freely confuse these two categories.
\end{chunk}

\begin{chunk}\label{remark:orthogonals} \textbf{Orthogonals and Verdier Quotients.} Given a class $\cat{C}$ of objects in a triangulated category $\cat{T}$, we write $\Thick(\cat{C})$ for the smallest thick subcategory of $\cat{T}$ containing $\cat{C}$. The full subcategories
\begin{align*}
\cat{C}^{\perp} &= \{ Y \in \cat{T} \l \Hom_{\cat{T}}(\Sigma^n X, Y) = 0 \text{ for all } X \in \cat{C} \text{ and } n \in \mathbb{Z} \}\\
{}^{\perp} \cat{C} &= \{ Y \in \cat{T} \l \Hom_{\cat{T}}(Y, \Sigma^n X) = 0 \text{ for all } X \in \cat{C} \text{ and } n \in \mathbb{Z} \}
\end{align*}
are called the \emph{right orthogonal} and \emph{left orthogonal} of $\cat{C}$, respectively. One checks that $\cat{C}^{\perp}$ is a colocalising subcategory of $\cat{T}$, equal to $\Thick(\cat{C})^{\perp}$. Similarly, ${}^{\perp} \cat{C}$ is a localising subcategory, equal to ${}^{\perp} \Thick(\cat{C})$. Given a triangulated subcategory $\cat{S} \subseteq \cat{T}$ and objects $M,N \in \cat{T}$ there is a canonical map
\[
\Hom_{\cat{T}}(M,N) \lto \Hom_{\cat{T}/\cat{S}}(M,N),
\]
which is an isomorphism if either $M \in {}^{\perp} \cat{S}$ or $N \in \cat{S}^{\perp}$ \cite[Lemma 9.1.5]{NeemanBook}. Let $\cat{S} \subseteq \cat{Q}$ be two thick subcategories of $\cat{T}$. The induced functor $\cat{Q}/\cat{S} \lto \cat{T}/\cat{S}$ is fully faithful and the essential image is the full subcategory of $\cat{T}/\cat{S}$ whose objects are those of $\cat{S}$. We will therefore typically identify $\cat{Q}/\cat{S}$ with a triangulated subcategory of $\cat{T}/\cat{S}$. Note that there is a canonical isomorphism $(\cat{T}/\cat{S})/(\cat{Q}/\cat{S}) \xlto{\sim} \cat{T}/\cat{Q}$.
\end{chunk}

\begin{chunk}\label{section:total_acyclicity} \textbf{Total acyclicity.} Given a triangulated category $\cat{T}$ and a class $\cat{C}$ of objects in $\cat{T}$, objects of the intersection $\cat{C}^{\perp} \cap {}^{\perp} \cat{C} \subseteq \cat{T}$ are said to be \emph{totally acyclic} (relative to $\cat{C}$). Let $A$ be a ring. Set $\cat{T} = \kproj{A}$ and let $\cat{C}$ be the class $\Proj(A)$ of projective $A$-modules. Then the objects of
\begin{align*}
\kprojl{\tac}{A} &:= (\Proj A)^{\perp} \cap {}^{\perp}(\Proj A)
\end{align*}
are the complexes of projective $A$-modules $P$ which are acyclic, and have the property that $\Hom_A(P,Q)$ is acyclic for every projective $A$-module $Q$. If we take $\cat{T} = \kinj{A}$ and $\cat{C} = \Inj(A)$ then the objects of
\[
\kinjl{\tac}{A} := (\Inj A)^{\perp} \cap {}^{\perp}(\Inj A)
\]
are the complexes of injective $A$-modules $I$ which are acyclic, and have the property that $\Hom_A(J, I)$ is acyclic for every injective $A$-module $J$. The analogue for complexes of flat $A$-modules is more subtle, and is the subject of Theorem \ref{theorem:description_tac}.
\end{chunk}

\begin{chunk}\label{section:review_of_pure} \textbf{The pure derived category of flat sheaves.} Let $(\cat{A}, \otimes)$ be an abelian category together with the structure of a symmetric monoidal category. A complex $F$ of objects in $\cat{A}$ is said to be \emph{pure acyclic} if it is acyclic and $F \otimes Y$ is acyclic for every $Y \in \cat{A}$.

Taking $\cat{A} = \qcmod{X}$ with the usual monoidal structure, this defines the class of pure acyclic complexes of sheaves. By \cite[Proposition 3.4]{Murfet07} a complex $F$ of flat sheaves is pure acyclic if and only if it is acyclic and every syzygy $Z^n(F)$ is flat. The pure acyclic complexes form a localising subcategory $\kflatl{\pac}{X}$ of $\kflat{X}$. The notation in the next definition reflects Neeman's contribution in establishing the fundamental theorems about the quotient $\kflat{X}/\kflatl{\pac}{X}$ for affine schemes $X$.

\begin{definition} The \emph{pure derived category of flat sheaves} on $X$ is the Verdier quotient
\[
\kprof{X} := \kflat{X}/\kflatl{\pac}{X}.
\]
The first author took a long time to settle on the notation for this category, which is referred to as the ``mock homotopy category of projectives'' and denoted $\K_m(\Proj X)$ in \cite{Murfet07}. He apologises for any confusion!
\end{definition}

The properties of $\kprof{X}$ are developed in \cite{Murfet07}, which is underpinned by Neeman's theorems in the affine case \cite{Neeman08, NeemanAdj}. We briefly recall the relevant results. The first thing to observe is that the construction is reasonable: $\kprof{X}$ is a triangulated category with coproducts and small Homs \cite[Theorem 3.16]{Murfet07}. 

Let $A$ be a ring. The homotopy category $\kproj{A}$ is a triangulated subcategory of $\kflat{A}$, and Neeman proves in \cite{Neeman08} that the right orthogonal of this subcategory is the full subcategory of pure acyclic complexes; that is,
\[
\kproj{A}^{\perp} = \kflatl{\pac}{A}.
\]
He proceeds to prove that the composite
\[
\kproj{A} \xlto{\inc} \kflat{A} \xlto{\can} \kflat{A}/\kflatl{\pac}{A}
\]
is an equivalence. This equivalence forms the basis for \cite{Murfet07}, since it suggests that the quotient $\kprof{X}$ may serve as a replacement for $\kproj{A}$ over non-affine schemes, where projective sheaves are rare. Neeman proves that if $A$ is coherent (in particular, if $A$ is noetherian) then $\kproj{A}$ is compactly generated. This was first proven, under an additional hypothesis, by J\o rgensen \cite{Jorgensen} who also proves that there is an equivalence
\[
\qderu{b}{\modd A}^{\textrm{op}} \xlto{\sim} \kproju{c}{A}.
\]
Let us now return to the general case. The category $\kprof{X}$ is compactly generated \cite[Theorem 4.10]{Murfet07} and closed symmetric monoidal with unit $\cat{O}_X$ \cite[\S 6]{Murfet07}. We denote the internal Hom by $\Map(-,-)$. Taking $\K$-flat resolutions defines a fully faithful functor $\bold{f}(-): \qderu{b}{\coh X} \lto \kprof{X}$ \cite[Remark 5.9]{Murfet07} and the composite
\[
\xymatrix@C+1pc{
\qderu{b}{\coh X}^{\textrm{op}} \ar[r]^{\bold{f}(-)^{\textrm{op}}} & \kprof{X}^{\textrm{op}} \ar[rr]^{\Map(-,\cat{O}_X)} & & \kprof{X}
}
\]
induces an equivalence \cite[Theorem 7.4]{Murfet07}
\[
\Phi: \qderu{b}{\coh X}^{\textrm{op}} \xlto{\sim} \kprofu{c}{X}.
\]
That is, $\Phi(G) = \Map(\bold{f}(G), \cat{O}_X)$. When $X$ has enough locally free sheaves of finite rank (e.g.\,$X$ is a quasi-projective variety) $\Phi$ sends a bounded complex $G$ of coherent sheaves to the dual $\shom(P,\cat{O}_X)$ of a resolution $P$ of $G$ by locally free sheaves of finite rank.
\end{chunk}

\section{Cotorsion flat sheaves}\label{section:cotorsion}

\begin{definition}\label{definition:cotorsion} A sheaf $\md{S}$ is called \emph{cotorsion} if $\Ext^1(\md{F}, \md{S}) = 0$ for any flat sheaf $\md{F}$. We say that $\md{S}$ is \emph{cotorsion flat} if it is both cotorsion and flat; the class of such sheaves is closed under finite direct sums and direct summands. 
\end{definition}

\begin{lemma}\label{lemma_appendix_qhomcotorsion} Let $\md{F}, \md{E}$ be sheaves with $\md{E}$ injective. Then
\begin{itemize}
\item[(i)] If $\md{F}$ is injective then  $\qhom(\md{F}, \md{E})$ is cotorsion flat.
\item[(ii)] If $\md{F}$ is flat then $\qhom(\md{F}, \md{E})$ is injective.
\end{itemize}
\end{lemma}
\begin{proof}
$(i)$ To prove flatness one uses a \v{C}ech argument to reduce to the affine case, which is well-known; see \cite[Lemma 8.7]{Murfet07}. For flat $\md{P}$, we have $\Ext^1(\md{P}, \qhom(\md{F}, \md{E})) \cong \Ext^1(\md{P} \otimes \md{F}, \md{E}) = 0$, so $\qhom(\md{F}, \md{E})$ is cotorsion. The proof of $(ii)$ is similar.
\end{proof}

We will see in Section \ref{section:tac} that cotorsion flat sheaves play the same role in $\kprof{X}$ that injective sheaves play in the derived category, namely, any flat sheaf can be resolved by cotorsion flat sheaves, and any morphism in $\kprof{X}$ whose codomain is a bounded below complex of cotorsion flat sheaves can be uniquely lifted to a homotopy equivalence class of cochain maps. To prove these two statements we will need the next proposition, which generalises to schemes a standard result in commutative algebra.

The proof of the proposition is delayed until the end of the section, since we must first develop some formal properties of $\qhom(-,-)$. Note that given a pair of sheaves $\md{F},\md{E}$ there is a canonical morphism $\alpha: \md{F} \lto \qhom(\qhom(\md{F}, \md{E}),\md{E})$. Indeed, such a standard morphism exists in any closed monoidal category; it is the adjoint partner of the twisted counit
\[
\md{F} \otimes \qhom(\md{F}, \md{E}) \cong \qhom(\md{F}, \md{E}) \otimes \md{F} \xlto{\varepsilon} \md{E}.
\]
Moreover, if $\md{F}$ is flat and $\md{E}$ is injective, then the sheaf $\qhom(\qhom(\md{F},\md{E}),\md{E})$ is cotorsion flat, by Lemma \ref{lemma_appendix_qhomcotorsion}.

\begin{proposition}\label{prop:cotorsion_flat_resolution} Given a flat sheaf $\md{F}$ and injective cogenerator $\md{E}$ for $\qcmod{X}$, the canonical morphism $\alpha$ fits into an exact sequence
\[
0 \lto \md{F} \xlto{\alpha} \qhom(\qhom(\md{F},\md{E}),\md{E}) \lto \md{F}' \lto 0
\]
with $\md{F}'$ a flat sheaf. Moreover, the following are equivalent:
\begin{itemize}
\item[(i)] $\md{F}$ is cotorsion.
\item[(ii)] $\alpha$ is a split monomorphism.
\item[(iii)] $\md{F}$ is a direct summand of $\qhom(\md{I}, \md{I}')$ for some injective sheaves $\md{I}, \md{I}'$.
\end{itemize}
\end{proposition}

We split the proof of the proposition into a series of lemmas.

\begin{lemma}\label{lemma_appendix_a_1} Let $f: U \lto X$ be the inclusion of an open subset. Given sheaves $\md{F}$ on $X$ and $\md{G}$ on $U$, there is a natural isomorphism $\qhom(\md{F}, f_*\md{G}) \xlto{\sim} f_* \qhom(\md{F}|_U, \md{G})$.
\end{lemma}
\begin{proof}
For a sheaf $\md{A}$ on $X$ we have
\begin{align*}
\Hom(\md{A}, \qhom(\md{F}, f_*\md{G})) &\cong \Hom(\md{A} \otimes \md{F}, f_*\md{G})\\
&\cong \Hom(\md{A}|_U \otimes \md{F}|_U, \md{G})\\
&\cong \Hom(\md{A}|_U, \qhom(\md{F}|_U, \md{G}))\\
&\cong \Hom(\md{A}, f_*\qhom(\md{F}|_U, \md{G}))
\end{align*}
which yields a natural isomorphism of the desired form.
\end{proof}

\begin{lemma}\label{lemma_projection_morphism} Let $f: U \lto X$ be the inclusion of an affine open subset. Given a coherent sheaf $\md{N}$ on $X$ and any sheaf $\md{F}$ on $U$, there is a natural isomorphism
\[
\beta: \md{N} \otimes f_*(\md{F}) \lto f_*(\md{N}|_U \otimes \md{F})
\]
\end{lemma}
\begin{proof}
The existence of such a natural morphism is clear, and in proving that $\beta$ is an isomorphism we may reduce to the case where $X$ is affine, so that there is an exact sequence $\cat{O}_X^{\oplus m} \lto \cat{O}_X^{\oplus n} \lto \md{N} \lto 0$ for some integers $m,n \ge 0$. From exactness of the functor $f_*: \qcmod{U} \lto \qcmod{X}$ (here we use the fact that $f$ is affine) we obtain a commutative diagram with exact rows
\[
\xymatrix{
\cat{O}_X^{\oplus m} \otimes f_*(\md{F}) \ar[d]_{\wr} \ar[r] & \cat{O}_X^{\oplus n} \otimes f_*(\md{F}) \ar[d]_{\wr} \ar[r] & \md{N} \otimes f_*(\md{F}) \ar[d]^{\beta} \ar[r] & 0\\
f_*( \cat{O}_X^{\oplus m}|_U \otimes \md{F}) \ar[r] & f_*( \cat{O}_X^{\oplus n}|_U \otimes \md{F} ) \ar[r] & f_*(\md{N}|_U \otimes \md{F}) \ar[r] & 0
}
\]
in which the first two columns are isomorphisms. It follows that $\beta$ is an isomorphism.
\end{proof}

\begin{lemma}\label{lemma:standard_iso_1} Given sheaves $\md{N}, \md{F}, \md{I}$ there is a natural morphism
\begin{equation}\label{eq:standard_iso_1}
\md{N} \otimes \qhom(\md{F}, \md{I}) \lto \qhom(\qhom(\md{N}, \md{F}), \md{I}).
\end{equation}
When $\md{N}$ is coherent and $\md{I}$ is injective this is an isomorphism.
\end{lemma}
\begin{proof}
The standard morphism (\ref{eq:standard_iso_1}) exists in any symmetric closed monoidal category: the counit $\varepsilon$ of the adjunction between the tensor and Hom provides a natural morphism (ignoring the intervention of twists)
\[
\qhom(\md{F}, \md{I}) \otimes \md{N} \otimes \qhom(\md{N}, \md{F}) \xlto{1 \otimes \varepsilon} \qhom(\md{F}, \md{I}) \otimes \md{F} \xlto{\varepsilon} \md{I}
\]
which is the adjoint partner of (\ref{eq:standard_iso_1}). Assume that $\md{N}$ is coherent and $\md{I}$ is injective and let $\mf{U} = \{ U_0, \ldots, U_d \}$ be an affine open cover of $X$. Consider the \v{C}ech resolution
\begin{equation}\label{eq:cech}
0 \lto \md{I} \lto \md{C}^0(\mf{U}, \md{I}) \lto \cdots \lto \md{C}^d(\mf{U}, \md{I}) \lto 0
\end{equation}
of $\md{I}$, where the sheaf $\md{C}^p(\mf{U}, \md{I})$ is the direct sum, over all sequences $i_0 < \cdots < i_p$ in the set $\{ 0,\ldots,d \}$ of length $p$, of sheaves $f_*(\md{I}|_{U_{i_0} \cap \cdots \cap U_{i_p}})$ where $f: U_{i_0} \cap \cdots \cap U_{i_p} \lto X$ denotes the inclusion. For any choice of sequence the functor $f_*$ has an exact left adjoint, and therefore sends injective sheaves to injective sheaves. Hence (\ref{eq:cech}) is an exact sequence of injective sheaves, which decomposes into a series of short (split) exact sequences.

Applying the functors $\md{N} \otimes \qhom(\md{F}, -)$ and $\qhom(\qhom(\md{N}, \md{F}), -)$ to these split exact sequences, we reduce to proving the lemma for injective sheaves of the form $f_*(\md{I})$ for some inclusion $f: U \lto X$ of an affine open subset, and $\md{I}$ an injective sheaf on $U$. But in that case we have isomorphisms
\begin{align*}
\md{N} \otimes \qhom(\md{F}, f_*\md{I}) &\cong \md{N} \otimes f_*\qhom(\md{F}|_U, \md{I}) &&\text{(Lemma \ref{lemma_appendix_a_1})}\\
&\cong f_*( \md{N}|_U \otimes \qhom(\md{F}|_U, \md{I})) &&\text{(Lemma \ref{lemma_projection_morphism})}
\end{align*}
and
\begin{align*}
\qhom(\qhom(\md{N}, \md{F}), f_*\md{I}) &\cong f_*\qhom(\qhom(\md{N}, \md{F})|_U, \md{I}) &&\text{(Lemma \ref{lemma_appendix_a_1})}\\
&\cong f_* \qhom(\qhom(\md{N}|_U, \md{F}|_U), \md{I}) &&\text{($\star$)}
\end{align*}
whence we reduce to the case where $X$ is affine, which is well-known (taking a presentation, one reduces to the case of $\md{N}$ finite-free, which is obvious). In the step marked $(\star)$ we use the fact that $\md{N}$ is coherent, so $\qhom(\md{N}, \md{F}) \cong \shom(\md{N}, \md{F})$ and consequently $\qhom(\md{N}, \md{F})|_U \cong \qhom(\md{N}|_U, \md{F}|_U)$.
\end{proof}

\begin{remark}\label{remark:global_sections_qhom} Given sheaves $\md{F}$ and $\md{G}$ there is an isomorphism
\begin{align*}
\Gamma(X, \qhom(\md{F}, \md{G})) &\cong \Hom(\cat{O}_X, \qhom(\md{F}, \md{G}))\\
&\cong \Hom(\cat{O}_X \otimes \md{F}, \md{G}) \cong \Hom(\md{F}, \md{G}).
\end{align*}
\end{remark}

A monomorphism $\phi: \md{F} \lto \md{F}'$ of sheaves is \emph{pure} if $\phi \otimes 1: \md{F} \otimes \md{A} \lto \md{F}' \otimes \md{A}$ is a monomorphism for every sheaf $\md{A}$. See \cite{Stenstrom} for the basic properties of purity.

\begin{lemma}\label{lemma:pureinjective_preenvelope} Let $\md{E}$ be an injective cogenerator for $\qcmod{X}$. For a sheaf $\md{F}$, the canonical morphism $\alpha: \md{F} \lto \qhom(\qhom(\md{F}, \md{E}), \md{E})$ is a pure monomorphism.
\end{lemma}
\begin{proof}
We first argue that $\alpha$ is a monomorphism. It suffices to prove that for every coherent sheaf $\md{N}$ the map $\Hom(\md{N}, \alpha)$ is injective. This map can be rewritten, up to isomorphism, as the composite
\begin{align*}
\Phi: \Hom(\md{N},\md{F}) \lto &\Hom(\md{N}, \qhom(\qhom(\md{F}, \md{E}), \md{E}))\\
\cong &\Hom(\md{N} \otimes \qhom(\md{F}, \md{E}), \md{E}) && \text{(Adjunction)}\\
\cong &\Hom(\qhom(\qhom(\md{N}, \md{F}), \md{E}), \md{E}). && \text{(Lemma \ref{lemma:standard_iso_1})}
\end{align*}
Suppose that we are given a nonzero morphism $\psi: \md{N} \lto \md{F}$, which corresponds to a nonzero morphism $\psi': \cat{O}_X \lto \qhom(\md{N}, \md{F})$. The fact that $\psi'$ is nonzero is witnessed by the injective cogenerator: we can find a morphism $\gamma: \qhom(\md{N}, \md{F}) \lto \md{E}$ with $\gamma \circ \psi' \neq 0$. Corresponding to $\gamma$ is a morphism of sheaves
\[
\gamma': \cat{O}_X \lto \qhom(\qhom(\md{N}, \md{F}), \md{E})
\]
and $\Phi(\psi) \circ \gamma': \cat{O}_X \lto \md{E}$ is the nonzero composite $\gamma \circ \psi'$. We conclude that $\Phi(\psi)$ is nonzero, so $\Phi$ is injective and $\alpha$ is a monomorphism. To prove that $\alpha = \alpha_{\md{F}}$ is pure, it is enough to prove that $\md{N} \otimes \alpha_{\md{F}}$ is a monomorphism for every coherent sheaf $\md{N}$. But we can identify $\md{N} \otimes \alpha_{\md{F}}$ with the morphism
\begin{align*}
\md{N} \otimes \md{F} \lto & \;\md{N} \otimes \qhom(\qhom(\md{F}, \md{E}), \md{E})\\
\cong &\; \qhom(\qhom(\md{N}, \qhom(\md{F}, \md{E})),\md{E}) && \text{(Lemma \ref{lemma:standard_iso_1})}\\
\cong &\; \qhom(\qhom(\md{N} \otimes \md{F}, \md{E}), \md{E}), && \text{(Adjunction)}
\end{align*}
which is $\alpha_{\md{N} \otimes \md{F}}$, known by the above to be a monomorphism. Hence $\alpha_{\md{F}}$ is pure.
\end{proof}

\begin{proof}[Proof of Proposition \ref{prop:cotorsion_flat_resolution}]
Given the flat sheaf $\md{F}$ we set $\md{S} = \qhom(\qhom(\md{F}, \md{E}),\md{E})$ where $\md{E}$ is an injective cogenerator. By Lemma \ref{lemma_appendix_qhomcotorsion} the sheaf $\md{S}$ is cotorsion flat, and by Lemma \ref{lemma:pureinjective_preenvelope} the canonical morphism $\alpha: \md{F} \lto \md{S}$ extends to a pure exact sequence
\[
0 \lto \md{F} \xlto{\alpha} \md{S} \lto \md{F}' \lto 0.
\]
It remains to argue that $\md{F}'$ is flat. Given $x \in X$, $\alpha_x: \md{F}_x \lto \md{S}_x$ is a pure monomorphism between flat $\cat{O}_{X,x}$-modules, and it is well-known that in this situation the cokernel $\md{F}'_x$ is flat \cite[I.11.1]{Stenstrom}, so $\md{F}'$ is flat. Finally, to prove $(i) \Rightarrow (ii)$ we observe that if $\md{F}$ is cotorsion then short exact sequence arising from $\alpha$ must split, and to prove $(ii) \Rightarrow (iii)$ and $(iii) \Rightarrow (i)$ we just apply Lemma \ref{lemma_appendix_qhomcotorsion}. 
\end{proof}

We conclude this section with two technical facts, needed in the sequel. The category $\qcmod{X}$ is Grothendieck abelian, and therefore complete. The product in this category is not the ordinary product of presheaves, and to avoid confusion we denote the product in $\qcmod{X}$ of a family $\{ \md{F}_i \}_{i \in I}$ by $\prod^{\qc}_{i \in I} \md{F}_i$. 

\begin{lemma}\label{lemma:cotorsion_flat_closed_products} The class of cotorsion flat sheaves is closed under arbitrary small products in $\qcmod{X}$.
\end{lemma}
\begin{proof}
Let $\md{E}$ be an injective cogenerator for $\qcmod{X}$ and $\{ \md{F}_i \}_{i \in I}$ a family of cotorsion flat sheaves. Each $\md{F}_i$ admits, by Proposition \ref{prop:cotorsion_flat_resolution}, a split monomorphism
\[
\md{F}_i \lto \qhom(\qhom(\md{F}_i, \md{E}), \md{E}),
\]
and taking the product in $\qcmod{X}$ we obtain a split monomorphism
\[
\prod^{\qc}_{i \in I} \md{F}_i \lto \prod^{\qc}_{i \in I} \qhom(\qhom(\md{F}_i, \md{E}), \md{E}) \cong \qhom\Big(\bigoplus_{i \in I} \qhom(\md{F}_i, \md{E}), \md{E}\Big).
\]
Any coproduct $\md{I}' = \oplus_{i \in I} \qhom(\md{F}_i, \md{E})$ of injective sheaves is injective, because $\qcmod{X}$ is locally noetherian \cite[V.4.3]{Stenstrom}, whence $\qhom(\md{I}', \md{E})$ is cotorsion flat (Lemma \ref{lemma_appendix_qhomcotorsion}). As a direct summand of a cotorsion flat sheaf, $\prod^{\qc}_{i \in I} \md{F}_i$ is cotorsion flat, as claimed.
\end{proof}

\begin{lemma}\label{lemma:technical_fact_coextendcot} Let $A \lto B$ be a faithfully flat morphism of rings, with $A$ noetherian. If $Q$ is a cotorsion flat $A$-module, then $\Hom_A(B, Q)$ is a cotorsion flat $B$-module.
\end{lemma}
\begin{proof}
It is straightforward to check that $\Hom_A(B, Q)$ is a cotorsion $B$-module, so let us prove only flatness. By Lemma \ref{lemma_appendix_qhomcotorsion} we may assume that $Q = \Hom_A(I,J)$ for some injective $A$-modules $I,J$. There is an isomorphism of $A$-modules
\[
\Hom_A(B, Q) = \Hom_A(B, \Hom_A(I,J)) \cong \Hom_A(B \otimes_A I, J).
\]
Since $B$ is flat over $A$, $B \otimes_A I$ is injective over $A$ (using Lazard's theorem, write $B$ as a direct limit of finite-free modules, and use the fact that $A$ is noetherian to see that direct limits of injectives are injective). It follows that $H := \Hom_A(B,Q)$ is a flat $A$-module. To see that $H$ is flat as a $B$-module, let an acyclic complex $T$ of $B$-modules be given. Then $(T \otimes_B H) \otimes_A B \cong T \otimes_{B} (H \otimes_A B)$ is acyclic, whence $T \otimes_{B} H$ is acyclic, because $A \lto B$ is faithfully flat.
\end{proof}

\section{N-totally acyclic complexes}\label{section:tac}

\begin{definition}\label{defn:total_acyclicity} A complex of flat sheaves $F$ is \emph{N-totally acyclic} if it is acyclic and~$F \otimes \md{I}$ is acyclic for every injective sheaf $\md{I}$. We write $\kflatl{\tac}{X}$ for the full subcategory of N-totally acyclic complexes in the homotopy category $\kflat{X}$ of flat sheaves.

Given a ring $A$ we define a complex of flat $A$-modules $F$ to be N-totally acyclic if the associated complex of quasi-coherent sheaves on $\Spec(A)$ has this property, that is, if $F$ is acyclic and $F \otimes_A I$ is acyclic for every injective $A$-module $I$. 
\end{definition}

Our choice of nomenclature is justified in Theorem \ref{theorem:description_tac}, where we show that N-totally acyclic complexes are precisely those complexes totally acyclic in the category $\kprof{X}$ with respect to the class of flat sheaves. Complexes of flat modules over a ring with this property were first studied by Enochs, Jenda and Torrecillas in \cite{Enochs93} and are referred to as \emph{complete flat resolutions} in \cite{Christensen00}. We begin by recording some basic facts from \cite[II \S 7]{ResiduesDuality} about injective sheaves on noetherian schemes; see also \cite[\S 2.1]{Conrad00}.

\begin{definition}\label{definition:injective_indecomp} Given $x \in X$ let $i_x: \Spec(\cat{O}_{X,x}) \lto X$ be the canonical map, $k(x)$ the residue field of $\cat{O}_{X,x}$ and $J(x)$ the quasi-coherent sheaf on $\Spec(\cat{O}_{X,x})$ associated to the injective envelope of $k(x)$. Then
\[
E(x) := (i_x)_* J(x)
\]
is an indecomposable injective sheaf on $X$.
\end{definition}

It is shown in \cite[II \S 7]{ResiduesDuality} that every injective sheaf $\md{I}$ is isomorphic to a coproduct $\oplus_{x \in X} E(x)^{\oplus \lambda_x}$ for some set of cardinals $\{\lambda_x\}_{x \in X}$, where $E(x)^{\oplus \lambda_x}$ denotes a $\lambda_x$-indexed coproduct of copies of $E(x)$. Moreover, the $\lambda_x$ are independent of the decomposition. It follows that a sheaf $\md{I}$ is injective if and only if $\md{I}_x$ is an injective $\cat{O}_{X,x}$-module for every $x \in X$. Using these facts, it is immediate that:

\begin{lemma}\label{lemma:acyclic_iff_tensor_indecomp} An acyclic complex $F$ of flat sheaves is N-totally acyclic if and only if $F \otimes E(x)$ is acyclic for every $x \in X$.
\end{lemma}

\begin{lemma}\label{lemma:total_acyc_pointwise} A complex $F$ of flat sheaves on $X$ is N-totally acyclic if and only if $F_x$ is an N-totally acyclic complex of flat $\cat{O}_{X,x}$-modules for every $x \in X$.
\end{lemma}
\begin{proof}
Suppose that $F$ is N-totally acyclic and let $x \in X$ and an injective $\cat{O}_{X,x}$-module $I$ be given. Let $\md{I}$ an injective sheaf on $X$ with stalk $I$ at $x$. By hypothesis $F \otimes \md{I}$ is acyclic, so $F_x \otimes I \cong F_x \otimes \md{I}_x \cong (F \otimes \md{I})_x$ is acyclic, and thus $F_x$ is N-totally acyclic. Conversely if $F$ is stalkwise N-totally acyclic then it must also have this property globally, because given an injective sheaf $\md{I}$ the $\cat{O}_{X,x}$-module $\md{I}_x$ is injective.
\end{proof}

\begin{lemma}\label{lemma:restriction_tac} Let $F$ be a complex of flat sheaves. Then
\begin{itemize}
\item[(i)] If $F$ is N-totally acyclic and $U \subseteq X$ is open, then $F|_U$ is N-totally acyclic.
\item[(ii)] If $\{ V_\alpha \}_{\alpha \in \Lambda}$ is an open cover of $X$ with the property that $F|_{V_{\alpha}}$ is N-totally acyclic for every $\alpha \in \Lambda$, then $F$ is N-totally acyclic.
\end{itemize}
\end{lemma}
\begin{proof}
Both claims are immediate from Lemma \ref{lemma:total_acyc_pointwise}.
\end{proof}

\begin{lemma}\label{lemma_shokrollah211} A complex $\md{L}$ of locally free sheaves of finite rank is N-totally acyclic if and only if both $\md{L}$ and $\shom(\md{L}, \cat{O}_X)$ are acyclic.
\end{lemma}
\begin{proof}
Both conditions are local, and it is well-known that for a local noetherian ring $A$ and a complex $L$ of finite free $A$-modules, $L$ is N-totally acyclic if and only if both $L$ and $\Hom_A(L,A)$ are acyclic; see \cite[Lemma 5.1.10]{Christensen00}.
\end{proof}

\begin{definition} A coherent sheaf $\md{G}$ is \emph{totally reflexive} if it is isomorphic to some syzygy $\md{G} \cong Z^n(\md{L})$ of an N-totally acyclic complex $\md{L}$ of locally free sheaves of finite rank. More generally, a sheaf $\md{F}$ is \emph{Gorenstein flat} if it is isomorphic to some syzygy of an N-totally acyclic complex of flat sheaves.
\end{definition}

\begin{lemma}\label{lemma:inclusion_affine} Let $f: U \lto X$ be the inclusion of an affine open subset, and let $F$ be an N-totally acyclic complex of flat sheaves on $U$. Then $f_* F$ is an N-totally acyclic complex of flat sheaves on $X$.
\end{lemma}
\begin{proof}
As $U$ is affine and $X$ is semi-separated, we reduce to the case where $X = \Spec(A)$ and $U = \Spec(B)$ are affine, and $f_*$ is restriction of scalars. If $F$ is an N-totally acyclic complex of flat $B$-modules, and $I$ an injective $A$-module, then $F \otimes_A I \cong F \otimes_B (I \otimes_A B)$. Since $I \otimes_A B$ is an injective $B$-module we deduce that $F \otimes_A I$ is acyclic, so $F$ is N-totally acyclic over $A$.
\end{proof}

We say that a flat ring morphism $\phi: A \lto B$ \emph{satisfies descent for N-total acyclicity} if any N-totally acyclic complex of flat modules over $B$ descends, via restriction of scalars, to an N-totally acyclic complex of flat modules over $A$.

\begin{proposition}\label{prop:descends_tac} If $A$ is a noetherian ring then the following morphisms satisfy descent for N-total acyclicity:
\begin{itemize}
\item[(i)] Any localisation $A \lto S^{-1}A$ at a multiplicatively closed set $S \subseteq A$.
\item[(ii)] Any faithfully flat ring morphism $A \lto B$ with $B$ noetherian.
\end{itemize}
\end{proposition}
\begin{proof}
We leave the proof of $(i)$ to the reader. $(ii)$ Let $E$ be an injective cogenerator for the category of $A$-modules, $F$ an N-totally acyclic complex of flat $B$-modules, and $I$ an injective $A$-module. Note that the $A$-module $Q = \Hom_A(I,E)$ is cotorsion flat, by Lemma \ref{lemma_appendix_qhomcotorsion}. We must prove that $F \otimes_A I$ is acyclic. But
\[
\Hom_A(F \otimes_A I, E) \cong \Hom_A(F, Q) \cong \Hom_{B}(F, \Hom_A(B, Q)).
\]
By Lemma \ref{lemma:technical_fact_coextendcot}, $\Hom_A(B,Q)$ is a cotorsion flat $B$-module, and by Lemma \ref{lemma_appendix_qhomcotorsion} it must be a direct summand of $\Hom_{B}(J, J')$ for some injective $B$-modules $J,J'$. Thus the complex $\Hom_A(F \otimes_A I, E)$ is a direct summand of
\[
\Hom_{B}(F, \Hom_{B}(J,J')) \cong \Hom_{B}(F \otimes_{B} J, J')
\]
which is acyclic by hypothesis. Since $E$ is an injective cogenerator, we infer that $F \otimes_A I$ is acyclic, and the proof is complete.
\end{proof}

\begin{example} Nontrivial examples of N-totally acyclic complexes can only exist when $X$ is singular: over a regular scheme every acyclic complex of flat sheaves is pure acyclic \cite[Proposition 9.7]{Murfet07}. Let $X$ be the curve $y^2 + x^{n+1} = 0$ defined over an algebraically closed field $k$, with $n$ even. There is an isolated singularity at $s = (0,0)$, and we set
\[
B := k[[x,y]]/(y^2 + x^{n+1}) \cong \widehat{\cat{O}_{X,s}}.
\]
There is a canonical morphism of schemes
\[
j: \Spec(B) \lto \Spec(\cat{O}_{X,s}) \lto X,
\]
and it follows from Lemma \ref{lemma:inclusion_affine} and Proposition \ref{prop:descends_tac} that the direct image functor $j_*$ sends N-totally acyclic complexes of flat $B$-modules to N-totally acyclic complexes of flat sheaves on $X$. Over $B$ there are many interesting N-totally acyclic complexes; for details of the following, see \cite[Chapter 5]{Yoshino90}. For any integer $0 \le \mu \le n/2$ the ideal $I_\mu = (y, x^\mu) \subseteq B$ is a maximal Cohen-Macaulay module over $B$ with complete projective resolution
\[
P_\mu: \cdots \lto B^{\oplus 2} \xlto{\rho} B^{\oplus 2} \xlto{\rho} B^{\oplus 2} \lto \cdots, \qquad \rho = \begin{pmatrix} y & x^\mu \\ x^{n-\mu} & -y\end{pmatrix}.
\]
Any complete projective resolution over $B$ is N-totally acyclic \cite[Proposition 5.1.4]{Christensen00}, so $j_*P_\mu$ is an N-totally acyclic complex of flat sheaves on $X$. This complex cannot be pure acyclic, because $I_\mu$ is the stalk of a syzygy in this complex, and $I_\mu$ is not flat.
\end{example}

\subsection{Total acyclicity in the pure derived category of flat sheaves} In this subsection we explain how N-total acyclicity can be understood as a total acyclicity condition, in the sense of Section \ref{section:total_acyclicity}, in the pure derived category $\kprof{X}$ of flat sheaves. The reader is referred to Section \ref{section:review_of_pure} for background material on this category.

\begin{lemma}\label{lemma:tac_is_triangulated} $\kflatl{\tac}{X}$ is a localising subcategory of $\kflat{X}$.
\end{lemma}
\begin{proof}
It is clear that N-total acyclicity is stable under shifts and homotopy equivalence. Given a triangle $F' \lto F \lto F'' \lto \Sigma F'$ in $\kflat{X}$ with $F',F$ N-totally acyclic, and an injective sheaf $\md{I}$, we have a triangle $F' \otimes \md{I} \lto F \otimes \md{I} \lto F'' \otimes \md{I} \lto \Sigma F' \otimes \md{I}$. From the long exact cohomology sequences of these two triangles we may conclude that $F''$ is N-totally acyclic. Finally, coproducts are exact and the tensor product commutes with coproducts, so any coproduct of N-totally acyclic complexes is N-totally acyclic.
\end{proof}

A complex of flat sheaves which is pure acyclic is clearly also N-totally acyclic. Hence, if a complex $F$ of flat sheaves is isomorphic in $\kprof{X}$ to an N-totally acyclic complex, then $F$ is also N-totally acyclic.

\begin{definition}\label{defn:ntac} Let $\kprofl{\tac}{X}$ denote the full subcategory of N-totally acyclic complexes in $\kprof{X}$. This is the essential image of the fully faithful functor
\[
\kflatl{\tac}{X}/\kflatl{\pac}{X} \hookrightarrow \kflat{X}/\kflatl{\pac}{X} = \kprof{X},
\]
so $\kprofl{\tac}{X}$ is a localising subcategory of $\kprof{X}$.
\end{definition}

The next lemma is crucial: it is the source of our interest in cotorsion flat sheaves.

\begin{lemma}\label{lemma:bounded_below_cot_orthog} Any bounded below complex of cotorsion flat sheaves belongs, as an object of $\kflat{X}$, to the orthogonal $\kflatl{\tac}{X}^{\perp}$.
\end{lemma}
\begin{proof}
Every bounded below complex of cotorsion flat sheaves is the homotopy limit in $\kflat{X}$ of a sequence of bounded complexes of cotorsion flat sheaves, so it suffices to prove that any cotorsion flat sheaf $\md{C}$ belongs to $\kflatl{\tac}{X}^{\perp}$. By Proposition \ref{prop:cotorsion_flat_resolution} we may assume that $\md{C} = \qhom(\md{I}, \md{I}')$ for injective sheaves $\md{I}, \md{I}'$. In this case, for any N-totally acyclic complex $E$ of flat sheaves, we have
\[
\Hom_{\kflat{X}}(E, \qhom(\md{I}, \md{I}')) \cong \Hom_{\K(X)}(E \otimes \md{I}, \md{I}')
\]
which vanishes, because $E \otimes \md{I}$ is acyclic and $\md{I}'$ is injective.
\end{proof}

\begin{remark}\label{remark:iso_kprof_iff_he} By the previous lemma, any bounded below complex $C$ of cotorsion flat sheaves belongs to $\ppe{X}^{\perp} \supseteq \kflatl{\tac}{X}^{\perp}$. It follows that for any complex $F$ of flat sheaves, the canonical map
\[
\Hom_{\kflat{X}}(F, C) \lto \Hom_{\kprof{X}}(F, C)
\]
is an isomorphism. In particular, two bounded below complexes $C,C'$ of cotorsion flat sheaves are isomorphic in $\kprof{X}$ if and only if they are homotopy equivalent.
\end{remark}

\begin{lemma}\label{lemma:flat_sheaf_is_cotorsion_cpx} Any flat sheaf is isomorphic, in $\kprof{X}$, to a bounded below complex of cotorsion flat sheaves which is unique, up to homotopy equivalence.
\end{lemma}
\begin{proof}
Given a flat sheaf $\md{F}$ we can, by repeated applications of Proposition \ref{prop:cotorsion_flat_resolution}, produce an exact sequence $E$ of flat sheaves of the form
\[
E: 0 \lto \md{F} \lto S^0 \lto S^1 \lto S^2 \lto \cdots
\]
with every $S^i$ cotorsion and $Z^i(E)$ flat for all $i \in \mathbb{Z}$. Such a complex $E$ is pure acyclic, by \cite[Proposition 3.4]{Murfet07}. From the triangle $\md{F} \lto S \lto E \lto \Sigma \md{F}$ in $\kflat{X}$ we infer the desired isomorphism $\md{F} \cong S$ in $\kprof{X}$. If $S'$ were another bounded below complex of cotorsion flat sheaves isomorphic in $\kprof{X}$ to $\md{F}$, then the isomorphism $S \cong S'$ would lift, by Remark \ref{remark:iso_kprof_iff_he}, to a homotopy equivalence between $S$ and $S'$.
\end{proof}

Let $\Flat(X)$ denote the class of flat sheaves, considered as objects of $\kprof{X}$. In the notation of Section \ref{remark:orthogonals}, $(\Flat X)^{\perp}$ is the full subcategory of $\kprof{X}$ consisting of objects $F$ such that $\Hom_{\kprof{X}}(\Sigma^i \md{P}, F) = 0$ for every flat sheaf $\md{P}$ and $i \in \mathbb{Z}$.

\begin{lemma}\label{lemma:hom_d_acyclic} $\kprofl{\ac}{X} = (\Flat X)^{\perp}$ as subcategories of $\kprof{X}$.
\end{lemma}
\begin{proof}
By \cite[Proposition 5.2]{Murfet07} flat sheaves are left orthogonal to acyclic complexes in $\kprof{X}$, so $\kprofl{\ac}{X} \subseteq (\Flat X)^{\perp}$. To prove the reverse inclusion, let $F$ be a complex of flat sheaves right orthogonal, in $\kprof{X}$, to arbitrary shifts of flat sheaves. We have to prove that $F$ is acyclic. 

There is a canonical functor $\kprof{X} \lto \qder{X}$ which is, up to equivalence, the quotient of $\kprof{X}$ by the full subcategory $\kprofl{\ac}{X}$ of acyclic complexes; see \cite[Remark 5.8]{Murfet07}.  For $\md{P} \in \Flat(X)$ and $i \in \mathbb{Z}$ we have
\begin{equation}\label{eq:tac_implie_ac1}
0 = \Hom_{\kprof{X}}(\md{P}, \Sigma^i F) \cong \Hom_{\qder{X}}(\md{P}, \Sigma^i F),
\end{equation}
because $\md{P}$ belongs to the orthogonal ${}^{\perp} \kprofl{\ac}{X}$. Let $F \lto I$ be the $\K$-injective resolution of $F$. We may continue (\ref{eq:tac_implie_ac1}) as follows:
\begin{equation}\label{eq:tac_implies_ac2}
\begin{split}
0 &= \Hom_{\qder{X}}(\md{P}, \Sigma^i F) \cong \Hom_{\qder{X}}(\md{P}, \Sigma^i I)\\
&\cong \Hom_{\K(X)}(\md{P}, \Sigma^i I) \cong H^i \Hom(\md{P}, I).
\end{split}
\end{equation}
Given $n \in \mathbb{Z}$ let $\phi: \md{P} \lto Z^n(I)$ be an epimorphism with $\md{P}$ a flat sheaf; such a morphism exists by \cite[Corollary 3.21]{Murfet07}. It follows from (\ref{eq:tac_implies_ac2}) that $\phi$ factors through the canonical morphism $I^{n-1} \lto Z^n(I)$ and in particular the composite $\md{P} \lto Z^n(I) \lto H^n(I)$ vanishes. Since $\phi$ is an epimorphism we conclude that $H^n(I) = 0$, which implies that both $I$ and $F$ must be acyclic, as claimed.
\end{proof}

\begin{remark}\label{remark:ac_colocalising} It follows from the lemma that $\kprofl{\ac}{X}$ is a colocalising subcategory of $\kprof{X}$. In particular, $\kprofl{\ac}{X}$ is closed under homotopy limits.
\end{remark}

In the next theorem, we write $[-,-]$ for morphism sets in $\kprof{X}$.

\begin{theorem}\label{theorem:description_tac} For a complex $F$ of flat sheaves the following are equivalent:
\begin{itemize}
\item[(i)] $F$ is N-totally acyclic.
\item[(ii)] $F$ is acyclic, and $\Hom(F, \md{C})$ is acyclic for every cotorsion flat sheaf $\md{C}$.
\item[(iii)] $F$ is left and right orthogonal in $\kprof{X}$ to shifts of flat sheaves. That is, for any flat sheaf $\md{P}$ and $i \in \mathbb{Z}$ we have
\[
0 = [F, \Sigma^i \md{P}] = [\Sigma^i \md{P}, F].
\]
\item[(iv)] $F$ is acyclic, and $\Map(F, \md{P})$ is acyclic for every flat sheaf $\md{P}$, where $\Map(-,-)$ denotes the internal Hom in $\kprof{X}$.
\end{itemize}
\end{theorem}
\begin{proof}
To elaborate on $(iv)$, note that $\kprof{X}$ is a closed monoidal category in which the tensor is the ordinary tensor product of complexes, and the internal Hom, denoted by $\Map(-,-)$, is defined by adjointness. The reader may find the details in \cite[\S 6]{Murfet07}. Alternatively, since we will only use (i)-(iii) in the sequel, $(iv)$ may be safely skipped.

$(i) \Rightarrow (ii)$ In proving that $\Hom(F, \md{C})$ is acyclic we may, by Proposition \ref{prop:cotorsion_flat_resolution}, assume that $\md{C} = \qhom(\md{I}, \md{I}')$ for injective $\md{I}$ and $\md{I}'$. We have an adjunction isomorphism
\[
\Hom(F \otimes \md{I}, \md{I}') \cong \Hom(F, \qhom(\md{I}, \md{I}')) = \Hom(F, \md{C}).
\]
The left hand-side is acyclic by hypothesis, and therefore so is the right hand-side.

$(ii) \Rightarrow (iii)$ Fix a flat sheaf $\md{P}$. From Lemma \ref{lemma:hom_d_acyclic} we deduce vanishing of $[\Sigma^i \md{P}, F]$ for every $i \in \mathbb{Z}$. It remains to prove that $[F, \Sigma^i \md{P}] = 0$. By Lemma \ref{lemma:flat_sheaf_is_cotorsion_cpx} the flat sheaf $\md{P}$ is isomorphic, in $\kprof{X}$, to a bounded below complex $S$ of cotorsion flat sheaves, and we infer from Remark \ref{remark:iso_kprof_iff_he} that there is an isomorphism
\[
[F, \Sigma^i \md{P}] \cong [F, \Sigma^i S] \cong \Hom_{\kflat{X}}(F, \Sigma^i S).
\]
The complex $S$ can be written in $\K(X)$ as the homotopy limit of bounded complexes, each of which is a finite extension of cotorsion flat sheaves, so to prove that $[F, \Sigma^i \md{P}] = 0$ it is enough to show that $\Hom_{\K(X)}(F, \Sigma^i \md{C}) \cong H^i \Hom(F, \md{C})$ vanishes for cotorsion flat $\md{C}$ and $i \in \mathbb{Z}$, which is true by hypothesis.

$(iii) \Rightarrow (i)$ It follows from Lemma \ref{lemma:hom_d_acyclic} that $F$ is acyclic. Let $\md{E}$ be an injective cogenerator for the category of sheaves, and let $\md{I}$ be an injective sheaf. By Lemma \ref{lemma_appendix_qhomcotorsion}, $\md{C} = \qhom(\md{I}, \md{E})$ is cotorsion flat, so using Remark \ref{remark:iso_kprof_iff_he} and the hypothesis gives
\begin{align*}
H^i\Hom(F \otimes \md{I}, \md{E}) &\cong \Hom_{\K(X)}(F \otimes \md{I}, \Sigma^i \md{E})\\
&\cong \Hom_{\kflat{X}}(F, \Sigma^i\md{C})\\
&\cong [F, \Sigma^i \md{C}] = 0
\end{align*}
for any $i \in \mathbb{Z}$. We infer that $F \otimes \md{I}$ is acyclic, whence $F$ is N-totally acyclic.

To conclude, we prove directly that $(i) \Leftrightarrow (iv)$. Firstly, we prove that $F$ is N-totally acyclic if and only if $\Map(F, \md{C})$ is acyclic for every cotorsion flat sheaf $\md{C}$. By Proposition \ref{prop:cotorsion_flat_resolution} we may assume that any such $\md{C}$ is of the form $\qhom(\md{I}, \md{E})$ for an injective sheaf $\md{I}$ and injective cogenerator $\md{E}$. Let $\md{P}$ be any shift of a flat sheaf, and notice that
\begin{align*}
[\md{P}, \Map(F, \md{C})] &\cong [\md{P} \otimes F, \qhom(\md{I},\md{E})] &&\quad \text{(Adjointness)}\\
&\cong \Hom_{\kflat{X}}(\md{P} \otimes F, \qhom(\md{I}, \md{E})) &&\quad \text{(Remark \ref{remark:iso_kprof_iff_he})}\\
&\cong \Hom_{\K(X)}(F \otimes \md{I}, \qhom(\md{P}, \md{E})). &&\quad \text{(Adjointness)}
\end{align*}
If $F$ is N-totally acyclic then $F \otimes \md{I}$ is acyclic, and as $\qhom(\md{P}, \md{E})$ is always injective, we have $[\md{P}, \Map(F, \md{C})] = 0$. Hence $\Map(F, \md{C})$ is acyclic, by Lemma \ref{lemma:hom_d_acyclic}. Conversely, if $\Map(F, \md{C})$ is acyclic for every cotorsion flat sheaf $\md{C}$, then we may take $\md{P} = \cat{O}_X$ in the above to see that $F$ is N-totally acyclic.

It is now clear that $(iv) \Rightarrow (i)$, and it remains to prove the reverse implication. Let $\md{P}$ be a flat sheaf which, by Lemma \ref{lemma:flat_sheaf_is_cotorsion_cpx}, is isomorphic in $\kprof{X}$ to a bounded below complex $S$ of cotorsion flat sheaves. We write $S$ as the homotopy limit, in $\K(X)$, of the sequence of its brutal truncations (denoted here by $\sigma_{\le n}S$)
\[
\cdots \lto \sigma_{\le 2} S \lto \sigma_{\le 1} S \lto \sigma_{\le 0} S.
\]
These are bounded complexes of cotorsion flat sheaves, and by Lemma \ref{lemma:cotorsion_flat_closed_products} any product, in $\K(X)$, of such complexes is a complex of cotorsion flat sheaves, so $S$ is the homotopy limit of the $\sigma_{\le n} S$ in $\kflat{X}$. The Verdier quotient functor
\[
\kflat{X} \lto \kprof{X}
\]
preserves products \cite[Theorem 5.5]{Murfet07} and consequently homotopy limits, so we have, finally, that $S$ is the homotopy limit of the $\sigma_{\le n} S$ in $\kprof{X}$. The functor $\Map(F,-)$ has a left adjoint and thus preserves homotopy limits, so $\Map(F,\md{P})$ is the homotopy limit of the complexes $\Map(F, \sigma_{\le n} S)$. The class of acyclic complexes in $\kprof{X}$ is closed under homotopy limits (Remark \ref{remark:ac_colocalising}), and every $\sigma_{\le n} S$ is a finite extension of cotorsion flat sheaves, so we will be done if we can prove that $\Map(F, \md{C})$ is acyclic for any cotorsion flat sheaf $\md{C}$. But this was accomplished above.
\end{proof}

\begin{remark} If $X = \Spec(A)$ is affine, and $P$ a complex of finitely generated projective $A$-modules, then $\Map(P, A)$ is isomorphic, in $\kprof{A}$, to the ordinary dual complex $\Hom_A(P,A)$ \cite[Corollary 6.13]{Murfet07}. 
\end{remark}

\begin{remark}\label{remark:orthogonals_intersecting} Condition $(iii)$ of Theorem \ref{theorem:description_tac} can be stated as an equality
\[
\kprofl{\tac}{X} = (\Flat X)^{\perp} \cap {}^{\perp}(\Flat X)
\]
of subcategories of $\kprof{X}$. In the terminology of Section \ref{section:total_acyclicity}, a complex of flat sheaves is N-totally acyclic precisely when it is totally acyclic in $\kprof{X}$ relative to the class of flat sheaves. This is analogous to the description, for a ring $A$, of the classes of totally acyclic complexes of projective and injective $A$-modules as intersections
\begin{gather*}
\kprojl{\tac}{A} = (\Proj A)^{\perp} \cap {}^{\perp}(\Proj A),\\
\kinjl{\tac}{A} = (\Inj A)^{\perp} \cap {}^{\perp}(\Inj A)
\end{gather*}
of orthogonals in the respective homotopy categories $\kproj{A}$ and $\kinj{A}$. In fact, this characterisation of $\kinjl{\tac}{A}$ is also valid over non-affine schemes; see \cite{Krause05}.
\end{remark}

\subsection{Two kinds of total acyclicity}\label{subsection:two_kinds} Let $A$ be a noetherian ring. Two notions of total acyclicity apply to a complex of projective $A$-modules: N-total acyclicity as a complex of flat $A$-modules, and total acyclicity as a complex of projective $A$-modules. Typically, these will agree:

\begin{lemma}\label{lemma:tac_iff_completeflat} Let $P$ be a complex of projective $A$-modules. Then
\begin{itemize}
\item[(i)] Suppose that $A$ has finite Krull dimension. If $P$ is totally acyclic, then it is also N-totally acyclic.
\item[(ii)] If $P$ is N-totally acyclic, then $\Hom_A(P,Q)$ is acyclic for every flat $A$-module $Q$. In particular, $P$ is totally acyclic.
\end{itemize}
\end{lemma}
\begin{proof}
(i) This is well-known; see \cite[Proposition 5.1.4]{Christensen00}\footnote{Note that this result is missing a hypothesis of finite Krull dimension, which is used in the proof.}. The proof in \emph{loc.cit.}\ uses the deep fact, due to Gruson and Jensen \cite[Part II, Corollary 3.2.7]{Raynaud71}, that over rings with finite Krull dimension flat modules have finite projective dimension.

(ii) Neeman has shown in \cite{Neeman08} that any complex of projective $A$-modules $P$ belongs to the orthogonal ${}^{\perp} \ppe{A}$ in $\kflat{A}$. We deduce that, for any flat $A$-module $Q$, there is an isomorphism (see Section \ref{remark:orthogonals})
\[
H^i \Hom_A(P,Q) \cong \Hom_{\kflat{A}}(P, \Sigma^i Q) \cong \Hom_{\kprof{A}}(P, \Sigma^i Q).
\]
This vanishes by Theorem \ref{theorem:description_tac}, so $P$ is totally acyclic. Note that for affine schemes, the proof of Theorem \ref{theorem:description_tac}$(i)\Rightarrow(iii)$ consists of standard facts about cotorsion flat modules; the only new ingredient is the aforementioned result of Neeman from \cite{Neeman08}.
\end{proof}

\begin{remark}\label{remark:comparison_flat_proj_tac} In \cite{Neeman08} Neeman proves that the inclusion $\kproj{A} \lto \kflat{A}$ has a right adjoint, and that the right orthogonal to $\kproj{A}$ consists of pure acyclic complexes. It follows that for any complex $F$ of flat $A$-modules there is a triangle
\[
P \lto F \lto E \lto \Sigma P
\]
in $\kflat{A}$ with $P$ a complex of projective $A$-modules and $E$ pure acyclic. It is clear that $F$ is N-totally acyclic if and only if $P$ is N-totally acyclic and, if $\dim(A) < \infty$, then Lemma \ref{lemma:tac_iff_completeflat} implies that this happens precisely when $P$ is totally acyclic.
\end{remark}

\begin{lemma}\label{lemma:equiv_ktac_ntac} If $A$ has finite Krull dimension then the equivalence
\[
\kproj{A} \xlto{\sim} \kprof{A}
\]
of Neeman \textup{(}see Section \ref{section:notation}\textup{)} restricts to an equivalence
\[
\kprojl{\tac}{A} \xlto{\sim} \kprofl{\tac}{A}.
\]
\end{lemma}
\begin{proof}
This follows from Lemma \ref{lemma:tac_iff_completeflat} and Remark \ref{remark:comparison_flat_proj_tac}.
\end{proof}

\subsection{Existence of adjoints}\label{section:existence_adjoints} Let $\cat{T}$ be a triangulated category with coproducts and $\cat{S}$ a localising subcategory of $\cat{T}$. The theory of Bousfield localisation (see \cite[Ch.\ 9]{NeemanBook}) states that the inclusion $\cat{S} \lto \cat{T}$ has a right adjoint if and only if there exists, for every $M \in \cat{T}$, a triangle $M' \lto M \lto M'' \lto \Sigma M'$ in $\cat{T}$, with $M' \in \cat{S}$ and $M'' \in \cat{S}^{\perp}$. Such a triangle is unique, up to isomorphism, and the right adjoint of the inclusion $\cat{S} \lto \cat{T}$ is defined by sending $M$ to the object $M'$ occurring in this triangle.

Let $A$ be a noetherian ring with a dualising complex. J\o rgensen proves in \cite{Jorgensen07} that the inclusion $\kprojl{\tac}{A} \lto \kproj{A}$ has a right adjoint; that is, for any complex $P$ of projective $A$-modules, there exists a triangle
\[
P' \lto P \lto P'' \lto \Sigma P'
\]
in $\kproj{A}$, with $P'$ totally acyclic, and $P''$ in the orthogonal $\kprojl{\tac}{A}^{\perp}$. J\o rgensen uses this triangle to deduce, among other things, the existence of Gorenstein projective precovers for arbitrary $A$-modules. In this section we return to the full generality of complexes of sheaves on a semi-separated noetherian scheme $X$, and give a generalisation of J\o rgensen's result, with no assumption on the existence of a dualising complex.

Recall that a \emph{homological functor} on a triangulated category $\cat{T}$ is an additive functor $H: \cat{T} \lto \ab$ which sends triangles in $\cat{T}$ to long exact sequences (here $\ab$ denotes the category of abelian groups). If $H$ sends small coproducts in $\cat{T}$ to coproducts in $\ab$, we say that $H$ is \emph{coproduct-preserving}. We define $\Ker(H)$ to be the full subcategory of objects $Q \in \cat{T}$ such that $H(\Sigma^n Q) = 0$ for every $n \in \mathbb{Z}$.

\begin{theorem}[Margolis]\label{theorem:margolis} Let $\cat{T}$ be a compactly generated triangulated category. For any coproduct-preserving homological functor $H$ on $\cat{T}$, the inclusion $\Ker(H) \lto \cat{T}$ has a right adjoint.
\end{theorem}
\begin{proof}
This was proved by Margolis \cite[\S 7]{Margolis83} in the category of spectra, and the same proof works in a compactly generated triangulated category; see \cite[\S 6]{Krause08}.
\end{proof}

\begin{theorem}\label{theorem:2} There is a coproduct-preserving homological functor on $\kprof{X}$ with kernel $\kprofl{\tac}{X}$, so the inclusion $\kprofl{\tac}{X} \lto \kprof{X}$ has a right adjoint.
\end{theorem}

\begin{proof} Using the notation of Definition \ref{definition:injective_indecomp}, take the coproduct $\md{I} = \bigoplus_{x \in X} E(x)$ of the set of indecomposable injective sheaves on $X$. The tensor product of an injective sheaf with a flat sheaf is injective, so there is a functor $- \otimes \md{I}: \kflat{X} \lto \kinj{X}$. It is straightforward to check that this functor vanishes on pure acyclic complexes; see \cite[Lemma 8.2]{Murfet07}. By definition $\kprof{X}$ is the Verdier quotient of $\kflat{X}$ by the full subcategory of pure acyclic complexes, so there is an induced functor
\[
- \otimes \md{I}: \kprof{X} \lto \kinj{X}.
\]
In particular, $H^0(- \otimes \md{I})$ gives a well-defined homological functor on $\kprof{X}$. Taking cohomology of objects in $\kprof{X}$ is also well-defined, so
\begin{gather*}
H: \kprof{X} \lto \ab,\\
H(F) := H^0( F ) \oplus H^0( F \otimes \md{I} ) = H^0(F) \oplus \bigoplus_{x \in X} H^0( F \otimes E(x) )
\end{gather*}
is a homological functor on $\kprof{X}$. By Lemma \ref{lemma:acyclic_iff_tensor_indecomp} the kernel of $H$ is $\kprofl{\tac}{X}$. The category $\kprof{X}$ is compactly generated \cite[Theorem 4.10]{Murfet07} so it follows from Theorem \ref{theorem:margolis} that the inclusion $\kprofl{\tac}{X} \lto \kprof{X}$ has a right adjoint.
\end{proof}

\begin{theorem}\label{theorem:jorgensen} If $A$ is a noetherian ring of finite Krull dimension then the inclusion
\[
\kprojl{\tac}{A} \lto \kproj{A}
\]
has a right adjoint, and every $A$-module has a Gorenstein projective precover.
\end{theorem}
\begin{proof}
The result follows from Theorem \ref{theorem:2} and the identifications of Lemma $\ref{lemma:equiv_ktac_ntac}$, but in the affine case the category $\kprof{X}$ is superfluous, so let us give a direct argument. Set $I = \oplus_{\mf{p} \in \Spec(A)} E(A/\mf{p})$ and consider the homological functor
\[
H^0(-) \oplus H^0(- \otimes I): \kproj{A} \lto \ab.
\]
The kernel is precisely the subcategory of N-totally acyclic complexes, which by Lemma \ref{lemma:tac_iff_completeflat} agrees with the subcategory $\kprojl{\tac}{A}$. Since $\kproj{A}$ is compactly generated \cite[Theorem 2.4]{Jorgensen}, we conclude by Theorem \ref{theorem:margolis} that the inclusion $\kprojl{\tac}{A} \lto \kproj{A}$ has a right adjoint. It now follows from the clever argument of J\o rgensen in \cite[\S 2]{Jorgensen07} (see also \cite[\S 5.6]{Buchweitz}) that Gorenstein projective precovers exist for $A$.
\end{proof}

\begin{remark} The proof of the Theorem \ref{theorem:jorgensen} is very similar to the proof of J\o rgensen's \cite[Proposition 1.9]{Jorgensen07}. We are able to avoid the hypothesis of a dualising complex in \emph{loc.cit.}\ by using Neeman's \cite{Neeman08} and working with N-totally acyclic complexes of flat modules, which are more suited to the argument than totally acyclic complexes of projectives (ultimately, because the defining condition of the former involves the tensor product, while the latter is defined using Hom).
\end{remark}

\begin{corollary}\label{corollary:ktac_wellgen} $\kprofl{\tac}{X}$ is well-generated.
\end{corollary}
\begin{proof}
The category $\kprof{X}$ is compactly generated by \cite[Theorem  4.10]{Murfet07}, and Theorem \ref{theorem:2} writes $\kprofl{\tac}{X}$ as the kernel of a homological functor on $\kprof{X}$, so the claim follows from \cite[Theorem 6.10]{Krause08}.
\end{proof}

\begin{corollary}\label{corollary:adjoint_ktacflat} The inclusion $\kflatl{\tac}{X} \lto \kflat{X}$ has a right adjoint.
\end{corollary}
\begin{proof}
By \cite[Theorem 3.16]{Murfet07} and Theorem \ref{theorem:2} the quotient functors
\[
\kflat{X} \lto \kprof{X}, \text{ and } \kprof{X} \lto \kprof{X}/\kprofl{\tac}{X}
\]
have fully faithful right adjoints, so their composite has a fully faithful right adjoint; since $\kflatl{\tac}{X}$ is the kernel of this composite, the inclusion $\kflatl{\tac}{X} \lto \kflat{X}$ must have a right adjoint by \cite[Lemma 3.2]{Krause05}.
\end{proof}

\begin{remark} The reader may wonder if Theorem \ref{theorem:2} can be applied, over a scheme, to prove the existence of some kind of precover. In a sequel \cite{MurfShokGor} to the present paper we use a modified form of J\o rgensen's argument \cite[\S 2]{Jorgensen07}, together with Theorem \ref{theorem:2}, to construct Gorenstein flat precovers.
\end{remark}

\subsection{Gorensteinness of schemes} Given a noetherian ring $A$ with dualising complex, Iyengar and Krause have proven that $A$ is Gorenstein if and only if every acyclic complex of projective $A$-modules is totally acyclic; see \cite[Corollary 5.5]{Krause06}. The next result gives the generalisation to schemes, plus an improvement on Iyengar and Krause's result in the affine case: we remove the dualising complex hypothesis.

In short, flat modules behave better than projectives, so we are able to pass from a local ring (which may not have a dualising complex) to its completion, which always admits a dualising complex; we then apply Iyengar and Krause's result to the completion.

\begin{theorem}\label{theorem:gorenstein_iff_actac} The scheme $X$ is Gorenstein if and only if every acyclic complex of flat sheaves is N-totally acyclic.
\end{theorem}
\begin{proof}
Suppose that $X$ is a Gorenstein scheme. Proving that every acyclic complex $F$ of flat sheaves is N-totally acyclic is a local question, so we may assume that $X = \Spec(A)$ is the spectrum of a local Gorenstein ring and $F$ an acyclic complex of flat $A$-modules. Every injective $A$-module $I$ has finite flat dimension, so by a standard argument $F \otimes_A I$ is acyclic and consequently $F$ is N-totally acyclic.

For the converse, let a point $x \in X$ and an acyclic complex $F$ of flat $\cat{O}_{X,x}$-modules be given. With $j: \Spec(\cat{O}_{X,x}) \lto X$ the canonical morphism, $j_* F$ is an acyclic complex of flat sheaves on $X$, which is N-totally acyclic by hypothesis. We conclude that $F \cong (j_* F)_x$ is N-totally acyclic; that is, every acyclic complex of $\cat{O}_{X,x}$-modules is N-totally acyclic. The upshot is that in proving the converse we may again assume that $X = \Spec(A)$ is the spectrum of a local noetherian ring $(A,\mf{m})$. Let $\widehat{A}$ denote the $\mf{m}$-adic completion.

An acyclic complex $F$ of flat $\widehat{A}$-modules is acyclic, and thus N-totally acyclic, as a complex of flat $A$-modules. Any injective $\widehat{A}$-module  $I$ is injective over $A$, so $F \otimes_A I$ is acyclic, from which we deduce that
\[
\widehat{A} \otimes_A (I \otimes_{\widehat{A}} F) \cong (I \otimes_A \widehat{A}) \otimes_{\widehat{A}} F \cong I \otimes_A F
\]
is acyclic. It follows that $F \otimes_{\widehat{A}} I$ is acyclic, as $\widehat{A}$ is faithfully flat. In particular, we may argue using Lemma \ref{lemma:tac_iff_completeflat} that every acyclic complex of projective $\widehat{A}$-modules is totally acyclic over $\widehat{A}$. Any complete local ring has a dualising complex, so from Iyengar and Krause's \cite[Corollary 5.5]{Krause06} we infer that $\widehat{A}$ is Gorenstein, whence $A$ is Gorenstein.
\end{proof}

\begin{corollary}\label{corollary:gorenstein_tac_ac_local} A noetherian ring $A$ of finite Krull dimension is Gorenstein if and only if every acyclic complex of projective $A$-modules is totally acyclic.
\end{corollary}
\begin{proof}
This follows from the theorem, together with Lemma \ref{lemma:tac_iff_completeflat} and Remark \ref{remark:comparison_flat_proj_tac}.
\end{proof}

\section{Two localisation sequences}\label{section:two_loc_seq}

In this section we study the Verdier quotient
\[
\qderl{\bold{G}}{X} := \kflat{X}/\kflatl{\tac}{X}.
\]
This is a triangulated category with coproducts, sitting between\footnote{In a sense which will be made precise in Remark \ref{remark:dg_sits_inbetween}.} the pure derived category of flat sheaves $\kprof{X}$ and the derived category $\qder{X}$. The relationship between $\qderl{\bold{G}}{X}$ and $\qder{X}$ involves the Gorenstein condition, as we will see in Corollary \ref{corollary:amusing} and various results of the next section; this explains the ``$\bold{G}$'' in the notation.

Our first task is to define some notation, to be used throughout the rest of this article. Both $\kprof{X}$ and $\qderl{\bold{G}}{X}$ are defined as Verdier quotients of $\kflat{X}$, and we write
\[
\pi: \kflat{X} \lto \kprof{X}, \quad \pi': \kflat{X} \lto \qderl{\bold{G}}{X}
\]
for the respective quotient functors. The composite
\[
\pi'': \kflat{X} \xlto{\inc} \K(X) \xlto{\can} \qder{X}
\]
vanishes on pure acyclic and totally acyclic complexes, and $\pi'$ vanishes on pure acyclic complexes. Using the universal property of the Verdier quotients, we obtain functors
\begin{align*}
\mu &: \kprof{X} \lto \qder{X},\\
\nu &: \qderl{\bold{G}}{X} \lto \qder{X},\\
\omega &: \kprof{X} \lto \qderl{\bold{G}}{X},
\end{align*}
which are, respectively, unique such that $\mu \circ \pi = \pi''$, $\nu \circ \pi' = \pi''$ and $\omega \circ \pi = \pi'$. Finally, we write $j: \kprofl{\tac}{X} \lto \kprof{X}$ for the inclusion.

To say that a pair $(G: \cat{T} \lto \cat{T}'',F: \cat{T}' \lto \cat{T})$ of triangulated functors is a \emph{localisation sequence} means that $F$ is fully faithful, that (up to equivalence) $G$ is the Verdier quotient of $\cat{T}$ by the image of $F$, and that $F$ has a right adjoint (equivalently, $G$ has a right adjoint). The functors $G,F$ and their adjoints are typically arranged in a diagram of the type given below; see \cite[Definition 3.1]{Krause05} for further details.

\begin{proposition}\label{prop:locseq_dg} The pair $(\omega, j)$ is a localisation sequence
\[
\xymatrix@C+1pc{
\kprofl{\tac}{X} \ar@<-0.8ex>[r]_{j} & \kprof{X} \ar@<-0.8ex>[l] \ar@<-0.8ex>[r]_{\omega} & \qderl{\bold{G}}{X} \ar@<-0.8ex>[l]
}
\]
and, in particular, $\omega$ induces an isomorphism $\kprof{X}/\kprofl{\tac}{X} \cong \qderl{\bold{G}}{X}$.
\end{proposition}
\begin{proof}
Using Section \ref{remark:orthogonals} we see that $\omega$ is, up to isomorphism, the quotient of $\kprof{X}$ by $\kprofl{\tac}{X}$. The inclusion $j$ of the latter category in the former has a right adjoint, by Theorem \ref{theorem:2}, so $(\omega, j)$ is a localisation sequence by \cite[Lemma 3.2]{Krause05}.
\end{proof}

The statements about well-generation in the next result and Corollary \ref{corollary:acovertac_smallhoms} will not be needed in the sequel, and may be safely skipped. We will see a different argument in the next section which shows that the involved categories are compactly generated when $X$ admits a dualising complex, and this covers most cases of interest.

\begin{corollary}\label{corollary:dg_smallhoms} $\qderl{\bold{G}}{X}$ has small Homs and is well-generated.
\end{corollary}
\begin{proof}
By the proposition $\omega$ has a fully faithful adjoint, so $\qderl{\bold{G}}{X}$ embeds as a subcategory of $\kprof{X}$. The latter category has small Homs \cite[Theorems 3.16]{Murfet07} so the same is true of $\qderl{\bold{G}}{X}$. We can write $\qderl{\bold{G}}{X}$ as the Verdier quotient of $\kprof{X}$, a compactly generated triangulated category \cite[Theorem 4.10]{Murfet07}, by a well-generated triangulated subcategory $\kprofl{\tac}{X}$ (Proposition \ref{corollary:ktac_wellgen}). It follows from the general theory of well-generated triangulated categories that $\qderl{\bold{G}}{X}$ is well-generated; see \cite{NeemanBook, Krause08}.
\end{proof}

\begin{remark}\label{remark:compact_notation} The inclusion $\kprofl{\ac}{X} \lto \kprof{X}$ induces a fully faithful functor
\[
\iota: \kprofl{\ac}{X}/\kprofl{\tac}{X} \hookrightarrow \kprof{X}/\kprofl{\tac}{X} \cong \qderl{\bold{G}}{X}.
\]
The essential image of $\iota$ is the full subcategory of acyclic complexes in $\qderl{\bold{G}}{X}$. Moreover, using Section \ref{remark:orthogonals} we see that the composite
\[
\kflatl{\ac}{X} \xlto{\pi|_{\ac}} \kprofl{\ac}{X} \xlto{\can} \kprofl{\ac}{X}/\kprofl{\tac}{X}
\]
induces an equivalence $\kflatl{\ac}{X}/\kflatl{\tac}{X} \cong \kprofl{\ac}{X}/\kprofl{\tac}{X}$.
\end{remark}

\begin{proposition}\label{prop:g_locseq} The pair $(\nu, \iota)$ is a localisation sequence
\[
\xymatrix@C+1pc{
\kprofl{\ac}{X}/\kprofl{\tac}{X} \ar@<-0.8ex>[r]_(0.7){\iota} & \qderl{\bold{G}}{X} \ar@<-0.8ex>[l] \ar@<-0.8ex>[r]_{\nu} & \qder{X}. \ar@<-0.8ex>[l]
}
\]
\end{proposition}
\begin{proof}
The canonical functor
\[
\gamma: \qderl{\bold{G}}{X} = \kflat{X}/\kflatl{\tac}{X} \lto \kflat{X}/\kflatl{\ac}{X}
\]
is a Verdier quotient, and from the equivalence $\delta: \kflat{X}/\kflatl{\ac}{X} \cong \qder{X}$ of \cite[Remark 5.8]{Murfet07} we infer that the composite $\nu = \delta \circ \gamma$ is also a Verdier quotient. To prove that $\nu$ has a right adjoint, we use a trick: consider the pair
\[
\kprof{X} \xlto{\omega} \qderl{\bold{G}}{X} \xlto{\nu} \qder{X}
\]
The composite $\nu \circ \omega = \mu$ has a right adjoint \cite[Theorem 5.5]{Murfet07} and since $\omega$ is a Verdier quotient we may infer from \cite[Lemma 5.5]{Tarrio00} that $\nu$ has a right adjoint, whence the pair $(\nu,\iota)$ is a localisation sequence by \cite[Lemma 3.2]{Krause05}.
\end{proof}

\begin{corollary}\label{corollary:acovertac_smallhoms} $\kprofl{\ac}{X}/\kprofl{\tac}{X}$ has small Homs and is well-generated.
\end{corollary}
\begin{proof}
The category in question certainly has small Homs, since it embeds in $\qderl{\bold{G}}{X}$. We can prove the well-generation statement in at least two ways: consider the functors
\begin{equation}\label{eq:aconvertac_smallhoms}
\kprofl{\tac}{X} \xlto{\inc} \kprofl{\ac}{X} \text{ and }
\qderl{\bold{G}}{X} \xlto{\nu} \qder{X}.
\end{equation}
The triangulated categories in (\ref{eq:aconvertac_smallhoms}) are well-generated, and $\kprofl{\ac}{X}/\kprofl{\tac}{X}$ is the Verdier quotient of the first functor and the kernel of the second; using either description it follows that the category is well-generated \cite{NeemanBook, Krause08}.
\end{proof}

We conclude this section with the following amusing consequence of the above.

\begin{corollary}\label{corollary:amusing} Consider the pair of functors
\[
\xymatrix@C+1pc{
\kprof{X} \ar[r]^{\omega} & \qderl{\bold{G}}{X} \ar[r]^{\nu} & \qder{X}.
}
\]
The scheme $X$ is Gorenstein if and only if $\nu$ is an equivalence, and $X$ is regular if and only if $\nu \circ \omega$ is an equivalence.
\end{corollary}
\begin{proof}
A Verdier quotient functor is an equivalence if and only if its kernel is the zero category. Since $\Ker(\nu) \cong \kprofl{\ac}{X}/\kprofl{\tac}{X}$ the first claim follows from Theorem \ref{theorem:gorenstein_iff_actac}. The composite $\nu \circ \omega$ is just the canonical functor $\mu: \kprof{X} \lto \qder{X}$, and by \cite[Theorem 5.5]{Murfet07} the functor $\mu$ is an equivalence if and only if $\kprofl{\ac}{X} = 0$. 

It remains to argue that this happens precisely when $X$ is regular. If $X$ is regular, then from \cite[Proposition 9.11]{Murfet07} with $U = \emptyset$ we infer that $\kprofl{\ac}{X} = 0$. For the converse, suppose that $\cat{O}_{X,x}$ is not regular for some $x \in X$. By \cite[Proposition 9.7]{Murfet07} there exists an acyclic complex $F$ of flat $\cat{O}_{X,x}$-modules which is not pure acyclic; applying $j_*$, where $j: \Spec(\cat{O}_{X,x}) \lto X$ is canonical, we obtain an acyclic complex $j_*F$ of flat sheaves which is not pure acyclic, whence $\kprofl{\ac}{X} \neq 0$.
\end{proof}

\section{Compact objects}\label{section:compact}

Assuming the existence of a dualising complex, we prove that both $\kprofl{\tac}{X}$ and $\kprofl{\ac}{X}/\kprofl{\tac}{X}$ are compactly generated and give a description in both cases of the subcategories of compact objects. In light of the properties of N-total acyclicity already established, the argument is a routine generalisation of \cite{Krause06} to schemes. One novelty is the treatment, in Section \ref{section:without_dualising}, of some schemes without a dualising complex.

\begin{setup} In this section we assume that $X$ admits a dualising complex $D$, which will be supposed to be a bounded complex of injective sheaves; see \cite{ResiduesDuality, Conrad00}.
\end{setup}

Let us first explain the role of the dualising complex, and Grothendieck duality. Recall that $D$ has coherent cohomology, and ``dualising by $D$'' defines an equivalence \cite{ResiduesDuality}
\begin{equation}\label{eq:defining_prop_dualising_cpx}
\rdev\qhom(-,D): \qderu{b}{\coh X}^{\textrm{op}} \xlto{\sim} \qderu{b}{\coh X}.
\end{equation}
Next we recall the main result of \cite[\S 8]{Murfet07}, which builds on earlier work of J\o rgensen \cite{Jorgensen}, Krause \cite{Krause05} and Neeman \cite{Neeman08}. The tensor product of an injective sheaf with a flat sheaf is injective, and a pure acyclic complex of flat sheaves tensored with an injective is contractible, so tensoring complexes of flat sheaves with $D$ defines a functor
\begin{equation}\label{eq:grothendieck_duality}
- \otimes\, D: \kprof{X} \lto \kinj{X},
\end{equation}
which turns out to be an equivalence with quasi-inverse $\qhom(D,-)$. The connection with Grothendieck duality is via two canonical equivalences:
\begin{align}
\Phi : \qderu{b}{\cmodn{X}}^{\textrm{op}} \xlto{\sim} \kprofu{c}{X}&, && \text{\cite[Theorem 7.4]{Murfet07}}\label{eq:two_equivalences_compacts}\\
\mathbf{i}(-): \qderu{b}{\cmodn{X}} \xlto{\sim} \kinju{c}{X}&. && \label{eq:two_equivalences_compacts_2} \text{\cite[Proposition 2.3]{Krause05}}
\end{align}
The functor $\mathbf{i}(-)$ sends a complex to its injective resolution, while $\Phi$ is slightly more complicated; a full description is given in Section \ref{section:review_of_pure}. The equivalences (\ref{eq:two_equivalences_compacts}) and (\ref{eq:two_equivalences_compacts_2}) fit into a diagram, commutative up to natural equivalence
\begin{equation}\label{eq:grothendieck_duality_2}
\xymatrix@C+4pc{
\kprofu{c}{X} \ar[r]^{- \otimes D}_{\sim} & \kinju{c}{X}\\
\qderu{b}{\cmodn{X}}^{\textrm{op}} \ar[u]_{\wr}^{\Phi}\ar[r]_{\rdev\qhom(-,D)}^{\sim} & \qderu{b}{\cmodn{X}}. \ar[u]^{\wr}_{\mathbf{i}(-)}
}
\end{equation}
In this sense, (\ref{eq:grothendieck_duality}) can be viewed as an analogue of Grothendieck duality for unbounded complexes. Inverting the rows of (\ref{eq:grothendieck_duality_2}), and noting that the bottom row is self-dual, one obtains for any $G \in \qderu{b}{\coh X}$ an isomorphism $\qhom(D, \bold{i}(G)) \cong \Phi \rdev\qhom(G,D)$.

The first proposition of this section uses (\ref{eq:grothendieck_duality}) to compare total acyclicity for complexes of flat and injective sheaves. This comparison was first established for noetherian rings by Iyengar and Krause in \cite{Krause06}. Recall that a complex $J$ of injective sheaves is \emph{totally acyclic} if it is acyclic, and $\Hom(\md{I},J)$ is acyclic for every injective sheaf $\md{I}$. We write $\kinjl{\tac}{X}$ for the full subcategory of totally acyclic complexes in $\kinj{X}$.

\begin{proposition}\label{prop:tac_iff_twoac} A complex $F$ of flat sheaves is N-totally acyclic if and only if both $F$ and $F \otimes D$ are acyclic, and this happens precisely when $F \otimes D$ is a totally acyclic complex of injective sheaves. Thus \textup{(}\ref{eq:grothendieck_duality}\textup{)} restricts to an equivalence
\[
- \otimes D: \kprofl{\tac}{X} \xlto{\sim} \kinjl{\tac}{X}.
\]
\end{proposition}
\begin{proof}
Let $\Flat(X)$ and $\Inj(X)$ be the classes of flat and injective sheaves, respectively, and consider the thick subcategories generated by these classes:
\begin{equation}\label{eq:tac_iff_twoac}
\Thick(\Flat X) \subseteq \kprof{X}, \text{ and } \Thick(\Inj X) \subseteq \kinj{X}.
\end{equation}
Because $- \otimes D$ sends $\Flat(X)$ into $\Thick(\Inj X)$, and $\qhom(D,-)$ sends $\Inj(X)$ into $\Thick(\Flat X)$, the thick subcategories in (\ref{eq:tac_iff_twoac}) are identified by the equivalence of (\ref{eq:grothendieck_duality}). This equivalence therefore identifies the subcategory
\[
(\Flat X)^{\perp} \cap {}^{\perp}(\Flat X) \subseteq \kprof{X}
\]
with the subcategory
\[
(\Inj X)^{\perp} \cap {}^{\perp}(\Inj X) \subseteq \kinj{X}.
\]
But by Theorem \ref{theorem:description_tac} the former subcategory is $\kprofl{\tac}{X}$ and the latter is $\kinjl{\tac}{X}$ by definition. Hence a complex $F$ of flat sheaves is N-totally acyclic if and only if $F \otimes D$ is totally acyclic. For the second claim, note that ${}^{\perp}(\Inj X) = \kinjl{\ac}{X}$, so a complex of flat sheaves $F$ belongs to ${}^{\perp}(\Flat X)$ if and only if $F \otimes D$ is acyclic. By Lemma \ref{lemma:hom_d_acyclic}
\[
\kprofl{\tac}{X} = (\Flat X)^{\perp} \cap {}^{\perp}(\Flat X) = \kprofl{\ac}{X} \cap \{ F \l F \otimes D \text{ is acyclic} \},
\]
which completes the proof.
\end{proof}

\begin{remark} Since the equivalence $\kprof{X} \cong \kinj{X}$ identifies the totally acyclic complexes on both sides, there is an equivalence $\qderl{\bold{G}}{X} \cong \kinj{X}/\kinjl{\tac}{X}$, and we are therefore justified in writing $\qderl{\bold{G}}{X}$ rather than $\qderl{\bold{G}}{\Flat X}$ or $\qderl{\bold{G}}{\Inj X}$.
\end{remark}

Next we recall a general construction. In the form given below, the result is \cite[Proposition 1.7]{Krause06}. One of the main inputs is the Neeman-Ravenel-Thomason localisation theorem, so the reader should also consult \cite{Neeman92,NeemanBook}.

\begin{construction}\label{remark:general_construction} Let $\cat{T}$ be a compactly generated triangulated category and $\cat{C}$ a class of compact objects in $\cat{T}$. Then $\cat{C}^{\perp}$ and $\cat{T}/(\cat{C}^{\perp})$ are compactly generated. Since $\cat{C}^{\perp}$ is localising and colocalising in $\cat{T}$ it follows that the inclusion $I: \cat{C}^{\perp} \lto \cat{T}$ has both a left adjoint $I_\lambda$ and right adjoint $I_\rho$ \cite[Theorem 4.1]{Neeman96}. Hence there is a recollement
\[
\xymatrix@C+2pc{
\cat{C}^{\perp} \ar[r]|{I} & \cat{T} \ar@<1.2ex>[l]\ar@<-1.2ex>[l] \ar[r]|(0.4){\can} & \cat{T}/(\cat{C}^{\perp}), \ar@<1.2ex>[l] \ar@<-1.2ex>[l]
}
\]
together with an equivalence up to direct summands
\[
\bold{s}: \cat{T}^c/\Thick(\cat{C}) \lto (\cat{C}^{\perp})^c
\]
and an equivalence
\[
\bold{t}: \Thick(\cat{C}) \xlto{\sim} (\cat{T}/\cat{C}^{\perp})^c.
\]
Let us describe how these equivalences are constructed.
The functor $I_\lambda$ has a coproduct preserving right adjoint, and therefore preserves compactness, by \cite[Theorem 5.1]{Neeman96}. The restriction $I_\lambda|_{\cat{T}^c}: \cat{T}^c \lto (\cat{C}^{\perp})^c$ vanishes on $\Thick(\cat{C})$, and the induced functor $\bold{s}: \cat{T}^c/\Thick(\cat{C}) \lto (\cat{C}^{\perp})^c$ is the desired equivalence up to direct summands.

The equivalence $\bold{t}$ is constructed as follows: let $\Loc(\cat{C})$ denote the smallest localising subcategory of $\cat{T}$ containing $\cat{C}$. This subcategory is compactly generated by the objects of $\cat{C}$, so $\Thick(\cat{C}) = \Loc(\cat{C})^c$. By the general theory of Bousfield localisation
\[
\Loc(\cat{C}) = {}^{\perp}(\Loc(\cat{C})^{\perp}) = {}^{\perp} (\cat{C}^{\perp}),
\]
and it follows that the composite
\[
\Loc(\cat{C}) \xlto{\inc} \cat{T} \xlto{\can} \cat{T}/(\cat{C}^{\perp})
\]
is an equivalence, which restricts to an equivalence $\bold{t}: \Thick(\cat{C}) = \Loc(\cat{C})^c \cong (\cat{T}/\cat{C}^{\perp})^c$.
\end{construction}

Following Iyengar and Krause \cite{Krause06} we apply Construction \ref{remark:general_construction} as follows: any thick subcategory $\cat{E} \subseteq \qderu{b}{\coh X}$ corresponds, under the equivalence $\Phi$ of (\ref{eq:two_equivalences_compacts}), to a thick subcategory $\cat{C} = \Phi \cat{E}$ of $\kprofu{c}{X}$. For a certain choice of thick subcategory $\cat{E}$, we will identify $\cat{C}^{\perp}$ with $\kprofl{\tac}{X}$, and use Construction \ref{remark:general_construction} to study $\kprofl{\tac}{X}$ and the quotient $\qderl{\bold{G}}{X} \cong \kprof{X}/\kprofl{\tac}{X}$.

\begin{definition}\label{definition:ssg} As described in the introduction, $\Perf(X)$ and $\Coperf(X)$ denote the full subcategories of $\qderu{b}{\cmodn{X}}$ consisting of the complexes of finite flat and finite injective dimension, respectively. It is clear that $\Coperf(X)$ is a triangulated subcategory.
\end{definition}

A complex of sheaves is called \emph{perfect} if it is locally isomorphic, in the derived category, to a bounded complex of locally free sheaves of finite rank. Recall from \cite{Neeman96} that a complex is perfect if and only if it is compact in $\qder{X}$.

\begin{lemma}\label{lemma:perfx_are_perfect} A bounded complex $G$ of coherent sheaves is perfect if and only if it has finite flat dimension. In particular, $\Perf(X)$ is a triangulated subcategory of $\qderu{b}{\cmodn{X}}$.
\end{lemma}
\begin{proof}
Finite flat dimension is a local property: by \cite[Proposition 1.1]{Tarrio97} or \cite[\S 3.2]{Murfet07} there is a bounded above complex $F$ of flat sheaves and a quasi-isomorphism $F \lto G$. Suppose that $G$ is locally of finite flat dimension. Then locally the syzygies of $F$ will be flat, in sufficiently high degrees; since $X$ is quasi-compact, we may find a single integer $N$ such that $Z^i(F)$ is flat for all $i < N$. Hence $G$ has finite flat dimension.

Perfection is also local, so we may reduce to the case of affine $X$. It is clear that a bounded complex of finitely generated modules has finite flat dimension if and only if it has finite projective dimension (i.e., is perfect).
\end{proof}

\begin{lemma}\label{lemma:phi_of_perfect} A complex $F$ of flat sheaves belongs to $\Phi \Perf(X)$ if and only if it is isomorphic, in $\kprof{X}$, to a bounded complex of flat sheaves with coherent cohomology.
\end{lemma}
\begin{proof}
It follows from \cite[Lemma 7.8]{Murfet07} that for any perfect complex $P$, $\Phi(P)$ is the $\K$-flat resolution by flat sheaves of the dual $\rdev\qhom(P,\cat{O}_X)$, which is perfect. But the $\K$-flat resolution of any complex is unique as an object of $\kprof{X}$ \cite[Remark 5.9]{Murfet07}, and any perfect complex has a finite resolution by flat sheaves, so $\Phi(P)$ is isomorphic in $\kprof{X}$ to a bounded complex of flat sheaves with coherent cohomology.

Conversely, let $F$ be a bounded complex of flat sheaves with coherent cohomology. By Lemma \ref{lemma:perfx_are_perfect}, $F$ is a perfect complex, so there is a quasi-isomorphism
\[
F \cong \rdev\qhom(\rdev\qhom(F, \cat{O}_X), \cat{O}_X).
\]
Applying \cite[Lemma 7.8]{Murfet07} for a second time, we find that $\Phi \rdev\qhom(F, \cat{O}_X)$ is the $\K$-flat resolution of $F$. But $F$ itself is $\K$-flat, so it is isomorphic in $\kprof{X}$ to its own $\K$-flat resolution, which is an object of $\Phi \Perf(X)$.
\end{proof}

We proved in Lemma \ref{lemma:hom_d_acyclic} that $\kprofl{\ac}{X}$ is equal, as a subcategory of $\kprof{X}$, to the orthogonal $(\Flat X)^{\perp}$. It is useful to reformulate this in terms of the right orthogonal of a class of \emph{compact} objects.

\begin{lemma}\label{lemma:right_orthog_ac} $\kprofl{\ac}{X} = (\Phi \Perf(X))^{\perp}$ as subcategories of $\kprof{X}$.
\end{lemma}
\begin{proof}
Consider the functor $\mu: \kprof{X} \lto \qder{X}$ of Section \ref{section:two_loc_seq}. It is shown in \cite[\S 5]{Murfet07} that $\mu$ is, up to equivalence, the Verdier quotient of $\kprof{X}$ by $\kprofl{\ac}{X}$, and that ${}^{\perp} \kprofl{\ac}{X}$ is the full subcategory of $\K$-flat complexes in $\kprof{X}$. By Lemma \ref{lemma:phi_of_perfect} the objects of $\Phi \Perf(X)$ are precisely the $\K$-flat resolutions of perfect complexes, and consequently for any $F \in \Phi \Perf(X)$ and $F' \in \kprof{X}$ we have an isomorphism
\[
\Hom_{\kprof{X}}(F, F') \cong \Hom_{\qder{X}}(F, F').
\]
It follows that $F'$ belongs to $(\Phi \Perf(X))^{\perp}$ if and only if $\Hom_{\qder{X}}(P, F') = 0$ for every perfect complex $P$. Since the perfect complexes compactly generate $\qder{X}$, we conclude that $\kprofl{\ac}{X} = (\Phi \Perf(X))^{\perp}$.
\end{proof}

As per our conventions (see Section \ref{section:notation}) $\Thick(\Perf(X), \Coperf(X))$ denotes the smallest thick subcategory of $\qderu{b}{\coh X}$ containing $\Perf(X) \cup \Coperf(X)$. The next observation is well-known.

\begin{lemma}\label{lemma:coperfect} A bounded complex of coherent sheaves $G$ belongs to $\Perf(X)$ \textup{(}respectively, $\Coperf(X)$\textup{)} if and only if $\rdev\qhom(G,D)$ belongs to $\Coperf(X)$ \textup{(}respectively, $\Perf(X)$\textup{)}. In particular, the category $\Thick(\Perf(X), \Coperf(X))$ is self-dual.
\end{lemma}
\begin{proof}
Here is a sketch: $D$ is $\K$-injective, so $\rdev\qhom(-,D) = \qhom(-,D)$. Applying $\qhom(-,D)$ to a bounded complex of flat sheaves produces, by Lemma \ref{lemma_appendix_qhomcotorsion}, a bounded complex of injective sheaves, and vice versa. It follows that the duality $\rdev\qhom(-,D)$ on $\qderu{b}{\coh X}$ restricts to a duality on $\Thick(\Perf(X), \Coperf(X))$.
\end{proof}

\begin{definition}\label{defn:our_thick_subcat} We define
\[
\cat{C} := \Phi\Thick(\Perf(X), \Coperf(X))
\]
to be the thick subcategory of $\kprofu{c}{X}$ corresponding, under the equivalence $\Phi$ of (\ref{eq:two_equivalences_compacts}), to the thick subcatgory $\Thick(\Perf(X), \Coperf(X))$ of $\qderu{b}{\coh X}$.
\end{definition}

\begin{proposition}\label{prop:right_orth_isktac} $\kprofl{\tac}{X} = \cat{C}^{\perp}$ as subcategories of $\kprof{X}$.
\end{proposition}
\begin{proof}
By Proposition \ref{prop:tac_iff_twoac}, a complex of flat sheaves $F$ is N-totally acyclic if and only if both $F$ and $F \otimes D$ are acyclic. We know from Lemma \ref{lemma:right_orthog_ac} that $\kprofl{\ac}{X} = (\Phi \Perf(X))^{\perp}$, so to complete the proof it suffices to show that $F$ belongs to $\Phi\Coperf(X)^{\perp}$ if and only if $F \otimes D$ is acyclic. The perfect complexes compactly generate $\qder{X}$, so $F \otimes D$ is acyclic if and only if $\Hom_{\qder{X}}(G, F \otimes D) = 0$ for every perfect complex $G$. Using \cite[Lemma 5.2]{Krause05} we see that $F \otimes D$ is acyclic if and only if
\[
\Hom_{\qder{X}}(G, F \otimes D) \cong \Hom_{\kinj{X}}(\mathbf{i}G, F \otimes D)
\]
vanishes for every perfect complex $G$, where $\mathbf{i}G$ denotes the injective resolution. Applying the inverse of the equivalence (\ref{eq:grothendieck_duality}) given at the beginning of this section and using commutativity of (\ref{eq:grothendieck_duality_2}), we obtain
\begin{align*}
\Hom_{\kinj{X}}(\mathbf{i}G, F \otimes D) &\cong \Hom_{\kprof{X}}(\qhom(D, \mathbf{i}G), F)\\
&\cong \Hom_{\kprof{X}}(\Phi\rdev\qhom(G, D), F).
\end{align*}
By Lemma \ref{lemma:coperfect} the objects of the form $\rdev\qhom(G, D)$ for some perfect complex $G$ are precisely the objects of $\Coperf(X)$, whence the claim.
\end{proof}

Taking $\cat{T} = \kprof{X}$ in Construction \ref{remark:general_construction} and using Proposition \ref{prop:right_orth_isktac}, we have
\begin{align*}
\cat{C}^{\perp} &= \kprofl{\tac}{X},
\end{align*}
and by Proposition \ref{prop:locseq_dg} there is a canonical isomorphism
\begin{align*}
\cat{T}/(\cat{C}^{\perp}) &\cong \qderl{\bold{G}}{X},
\end{align*}
so as a special case of the construction we obtain the following three propositions. In the statements we make use of the functors $j, \omega, \iota,\nu$ defined in Section \ref{section:two_loc_seq}, and adopt the convention that given a functor $F$, the left adjoint (if it exists) is denoted $F_\lambda$.

\begin{proposition}\label{prop:compacts_ktac} $\kprofl{\tac}{X}$ is compactly generated, and the composite
\[
\xymatrix{
\qderu{b}{\coh X}^{\textrm{op}} \ar[r]^{\Phi}_{\sim} & \kprofu{c}{X} \ar[r]^(0.46){j_\lambda} & \kproful{c}{\tac}{X}
}
\]
induces an equivalence up to direct summands
\begin{equation}\label{eq:compacts_ktac}
\qderu{b}{\cmodn{X}}/\Thick(\Perf(X), \Coperf(X))^{\textrm{op}} \lto \kproful{c}{\tac}{X}.
\end{equation}
\end{proposition}

Observe that (\ref{eq:compacts_ktac}) is defined in terms of the adjoint $j_\lambda$, which exists by a representability argument and is therefore somewhat mysterious. In general, one probably cannot hope to write down explicitly the N-totally acyclic complex corresponding to any given complex of coherent sheaves, but in certain cases, for example over Gorenstein schemes, there is a description in terms of the complete resolution.

\begin{proposition}\label{prop:qderlg_compact} $\qderl{\bold{G}}{X}$ is compactly generated, and the composite
\[
\xymatrix{
\qderu{b}{\coh X}^{\textrm{op}} \ar[r]^{\Phi}_{\sim} & \kprofu{c}{X} \ar[r]^{\inc} & \kprof{X} \ar[r]^(0.55){\omega} & \qderl{\bold{G}}{X}
}
\]
restricts to an equivalence $\Thick(\Perf(X), \Coperf(X))^{\textrm{op}} \xlto{\sim} \qderul{c}{\bold{G}}{X}$.
\end{proposition}

Implicit in the statement of the next proposition is the existence of both left and right adjoints to the inclusion $j: \kprofl{\tac}{X} \lto \kprof{X}$. We have already established the existence of a right adjoint in Section \ref{section:existence_adjoints}, in greater generality, but the existence of the left adjoint is new. The construction of the left adjoint given here depends crucially on the existence of a dualising complex, but in Section \ref{section:without_dualising} we will give a proof which avoids this hypothesis, for a special class of schemes. It would be interesting to know if the left adjoint exists in general.

\begin{proposition}\label{prop:recollement_dg} The pair $(\omega, j)$ is a recollement
\[
\xymatrix@C+1pc{
\kprofl{\tac}{X} \ar[r]|(0.55){j} & \kprof{X} \ar@<1.2ex>[l]\ar@<-1.2ex>[l] \ar[r]|(0.55){\omega} & \qderl{\bold{G}}{X}. \ar@<1.2ex>[l] \ar@<-1.2ex>[l]
}
\]
\end{proposition}

Next, we turn our attention to the Verdier quotient $\kprofl{\ac}{X}/\kprofl{\tac}{X}$.

\begin{proposition}\label{prop:g_recollement} The pair $(\nu, \iota)$ is a recollement
\[
\xymatrix@C+1pc{ 
\kprofl{\ac}{X}/\kprofl{\tac}{X} \ar[r]|(0.7){\iota} & \qderl{\bold{G}}{X} \ar@<1.2ex>[l]\ar@<-1.2ex>[l] \ar[r]|(0.55)\nu & \qder{X} \ar@<1.2ex>[l] \ar@<-1.2ex>[l]
}
\]
\end{proposition}
\begin{proof}
The localisation sequence $(\nu,\iota)$ was defined in Section \ref{section:two_loc_seq}. To prove that this pair is a recollement, we produce a left adjoint for $\nu$. Consider the following pair of functors
\[
\kprof{X} \xlto{\omega} \qderl{\bold{G}}{X} \xlto{\nu} \qder{X},
\]
using the notation introduced in Section \ref{section:two_loc_seq}. The first is a Verdier quotient by Proposition \ref{prop:locseq_dg}, so $\nu \circ \omega$ has a left or right adjoint if and only if $\nu$ does; see \cite[Lemma 5.5]{Tarrio00}. But by \cite[Theorem 5.5]{Murfet07} the functor $\mu = \nu \circ \omega$ has a left adjoint.
\end{proof}

\begin{proposition}\label{prop:compacts_dgac} The quotient $\kprofl{\ac}{X}/\kprofl{\tac}{X}$ is compactly generated, and the composite
\[
\xymatrix{
\Thick(\Perf(X),\Coperf(X))^{\textrm{op}} \ar[r]^(0.65){\sim}_(0.65){(\ref{prop:qderlg_compact})} & \qderul{c}{\bold{G}}{X} \ar[r]^(0.28){\iota_\lambda} & \left( \kprofl{\ac}{X}/\kprofl{\tac}{X} \right)^c 
}
\]
induces an equivalence up to direct summands
\[
\Big( \Thick(\Perf(X), \Coperf(X))/\Perf(X) \Big)^{\textrm{op}} \lto \left( \kprofl{\ac}{X}/\kprofl{\tac}{X} \right)^c.
\]
\end{proposition}
\begin{proof}
To be clear, $\iota_\lambda$ is the restriction to compact objects of the left adjoint of the fully faithful functor $\iota$ occurring in Proposition \ref{prop:g_recollement} above. The proof is an application of the Neeman-Ravenel-Thomason localisation theorem. Set $\cat{B} = \Phi \Perf(X)$. By Lemma \ref{lemma:right_orthog_ac} and Proposition \ref{prop:right_orth_isktac} we have
\[
\cat{B}^{\perp}/\cat{C}^{\perp} = \kprofl{\ac}{X}/\kprofl{\tac}{X}.
\]
We conclude from \cite[Proposition 1.7]{Krause06} that the quotient
is compactly generated, and we also obtain the desired equivalence up to direct summands.
\end{proof}

\begin{remark}\label{remark:dg_sits_inbetween} From Proposition \ref{prop:recollement_dg} and Proposition \ref{prop:g_recollement} we learn that the functors $\omega$ and $\nu$ have fully faithful left adjoints $\omega_\lambda$ and $\nu_\lambda$, which give a pair of inclusions
\[
\xymatrix@C+1pc{
\qder{X} \ar[r]^{\nu_\lambda} & \qderl{\bold{G}}{X} \ar[r]^(0.45){\omega_\lambda} & \kprof{X}.
}
\]
These functors preserve compactness, and there is a commutative diagram
\[
\xymatrix@C+1.8pc{
\qderu{c}{X} \ar[r]^{\nu_\lambda} & \qderul{c}{\bold{G}}{X} \ar[r]^{\omega_\lambda} & \kprofu{c}{X}\\
\Perf(X)^{\textrm{op}} \ar[u]_{\wr}^{\rdev\qhom(-,\cat{O}_X)} \ar[r]_(0.35){\inc} & \Thick(\Perf(X), \Coperf(X))^{\textrm{op}} \ar[u]_{\wr}^{(\ref{prop:qderlg_compact})} \ar[r]_(0.58){\inc} & \qderu{b}{\coh X}^{\textrm{op}} \ar[u]_{\wr}^{\Phi}
}
\]
\end{remark}

\subsection{Without a dualising complex}\label{section:without_dualising} In this subsection we drop the assumption that $X$ admits a dualising complex, which was used in Proposition \ref{prop:recollement_dg} to prove that the inclusion $\kprofl{\tac}{X} \lto \kprof{X}$ has a left adjoint. Here we prove, independently of the existence of a dualising complex, that this adjoint exists for certain schemes. By a standard argument, given in the next lemma, it suffices to prove that $\kprofl{\tac}{X}$ is closed under products in $\kprof{X}$. 

\begin{lemma}\label{lemma:left_adjoint_iff_products} The following conditions are equivalent:
\begin{itemize}
\item[(i)] $\kprofl{\tac}{X}$ is closed under small products in $\kprof{X}$.
\item[(ii)] The inclusion $j: \kprofl{\tac}{X} \lto \kprof{X}$ has a left adjoint.
\end{itemize}
If these equivalent conditions are satisfied, then $\kprofl{\tac}{X}$ is compactly generated.
\end{lemma}

\begin{remark}\label{remark:products_in_nflat} Products exist in $\kprof{X}$, because products exist in any compactly generated triangulated category, by \cite{Neeman98,NeemanBook}. Suppose that $X = \Spec(A)$ is affine. Then products of flat $A$-modules are flat and the quotient $\pi: \kflat{A} \lto \kprof{A}$ preserves products \cite{Neeman08}, so the product in $\kprof{A}$ is just the ordinary degree-wise product of complexes. Over general schemes the functor $\pi$ does not necessarily preserve products \cite[Remark A.15]{Murfet07}, so the product in $\kprof{X}$ is more complicated.
\end{remark}
\begin{proof}
Consider the localisation sequence $(\omega, j)$ of Proposition \ref{prop:locseq_dg}. If $\kprofl{\tac}{X}$ is closed under products then $\omega$ preserves products, and as $\kprof{X}$ is compactly generated we deduce from the dual Brown representability theorem \cite[Theorem 8.6.1]{NeemanBook} that $\omega$ has a left adjoint. It follows from the dual of \cite[Lemma 3.2]{Krause05} that $j$ has a left adjoint. Conversely, if $j$ has a left adjoint then $\omega$ has a left adjoint, so the kernel $\kprofl{\tac}{X}$ of $\omega$ is closed under products.

Finally, if the inclusion $j$ has a left adjoint $j_\lambda$, then this adjoint preserves compactness, as $j$ preserves coproducts \cite[Theorem 5.1]{Neeman96}. It is then easy to check that $j_\lambda$ sends a compact generating set of $\kprof{X}$ to a compact generating set for $\kprofl{\tac}{X}$.
\end{proof}

Recall that a point $x \in X$ is \emph{Gorenstein} if the local ring $\cat{O}_{X,x}$ is Gorenstein. 

\begin{proposition}\label{prop:isolated_gorenstein} If every non-closed point of $X$ is Gorenstein \textup{(}for example, if $X$ has only isolated singularities\textup{)} then $\kprofl{\tac}{X}$ is closed under products in $\kprof{X}$.
\end{proposition}
\begin{proof}
We want reduce to the affine case, where we understand the product. Let $\{ F_i \}_{i \in I}$ be a family of N-totally acyclic complexes of flat sheaves. Using \v{C}ech complexes, each $F_i$ is an extension of direct sums of complexes $f_*(F_i|_U)$, where $f: U \lto X$ is the inclusion of an affine open subset \cite[Proposition 3.13]{Murfet07}. Taking the product, in $\kprof{X}$, of the triangles involved in these extensions over the index set $I$, we deduce that $\prod_i F_i$ belongs to the triangulated subcategory of $\kprof{X}$ generated by products of the form $\prod_i f_*(F_i|_U)$. Products and direct image commute, i.e. $\prod_i f_*(F_i|_U) \cong f_* \prod_i(F_i|_U)$, and $f_*$ preserves N-total acyclicity by Lemma \ref{lemma:inclusion_affine}, so we can reduce to the case where $X = \Spec(A)$ is affine. Keeping in mind Remark \ref{remark:products_in_nflat}, the affine case is handled by the subsequent two results.
\end{proof}

\begin{lemma}\label{lemma:products_cpxs_flats} Let $A$ be a noetherian ring and $\mf{m} \subseteq A$ a maximal ideal. Given a family $\{ F_i \}_{i \in I}$ of complexes of flat $A$-modules, there is a canonical homotopy equivalence
\begin{equation}\label{eq:isolated_gorenstein}
\left(\prod_i F_i\right) \otimes_A E(A/\mf{m}) \cong \Gamma_\mf{m}\prod_i( F_i \otimes_A E(A/\mf{m}) ).
\end{equation}
where $\Gamma_\mf{m}(M) = \{ x \in M \l \mf{m}^n x = 0 \text{ for some $n > 0$} \}$ is the usual support functor.
\end{lemma}
\begin{proof}
By \cite[Theorem 1.2, Lemma 5.3]{Neeman08} every $F_i$ is isomorphic in $\kprof{A}$ to a complex of free $A$-modules. Taking the tensor product with $E(A/\mf{m})$ defines a functor on $\kprof{A}$ \cite[Facts 2.17]{Neeman08} so in proving the lemma, we may assume that for $i \in I$ and $j \in \mathbb{Z}$, $F^j_i$ is free of rank $\lambda_{i,j}$. We write $E(A/\mf{m})$ as the direct limit of submodules
\[
A_n = \{ a \in E(A/\mf{m}) \l \mf{m}^n \cdot a = 0 \}.
\]
The $A_n$ are finitely generated $A$-modules \cite{Matlis58}, so the functors $A_n \otimes_A -$ commute with products. It follows that there is an isomorphism
\begin{align*}
\Gamma_\mf{m}\prod_i( F_i \otimes_A E(A/\mf{m}) )^j &= \{ \text{$\mf{m}$-torsion elements in } \prod_i F_i^j \otimes_A E(A/\mf{m}) \}\\
&\cong \{ \text{$\mf{m}$-torsion elements in } \prod_i E(A/\mf{m})^{\oplus (\lambda_{i,j})} \}\\
&= \varinjlim_{n \ge 1} \prod_i A_n^{\oplus (\lambda_{i,j})} \cong \varinjlim_{n \ge 1} \prod_i ( F_i^j \otimes_A A_n )\\
&\cong \varinjlim_{n \ge 1}\left(\prod_i F_i^j\right) \otimes_A A_n \cong \left(\prod_i F_i\right)^j \otimes_A E(A/\mf{m}).
\end{align*}
One checks that this isomorphism commutes with differentials, giving the desired isomorphism (\ref{eq:isolated_gorenstein}) in the homotopy category of complexes of $A$-modules.
\end{proof}

\begin{proposition} Let $A$ be a noetherian ring with the property that $A_{\mf{p}}$ is Gorenstein for every non-maximal prime ideal $\mf{p}$. If $\{ F_i \}_{i \in I}$ is a family of N-totally acyclic complexes of flat $A$-modules, then $\prod_i F_i$ is N-totally acyclic. In particular, an arbitrary product of Gorenstein flat $A$-modules is Gorenstein flat.
\end{proposition}
\begin{proof}
Applying Theorem \ref{theorem:gorenstein_iff_actac} at the non-maximal primes, we learn that a complex $F$ of flat $A$-modules is N-totally acyclic if and only if it is acyclic and $F \otimes_A E(A/\mf{m})$ is acyclic for every maximal ideal $\mf{m}$. Fix a maximal ideal $\mf{m}$. The product $\prod_i F_i$ is an acyclic complex of flat $A$-modules, and by Lemma \ref{lemma:products_cpxs_flats} there is a homotopy equivalence
\[
\left(\prod_i F_i\right) \otimes_A E(A/\mf{m}) \cong \Gamma_\mf{m}\prod_i( F_i \otimes_A E(A/\mf{m}) ).
\]
The $F_i$ are N-totally acyclic, so every $F_i \otimes_A E(A/\mf{m})$ is acyclic and consequently $I = \prod_i( F_i \otimes_A E(A/\mf{m}) )$ is an acyclic complex of injective $A$-modules. It follows that $\Gamma_\mf{m}(I)$ is acyclic: brutally truncating $I$ from below in any degree gives an injective resolution of some syzygy $M$, and the cohomology of $\Gamma_\mf{m}(I)$ above the truncation point calculates the local cohomology $H^i_\mf{m}(M)$, which vanishes outside a finite range (not depending on $M$). We deduce that $(\prod_i F_i) \otimes_A E(A/\mf{m}) \cong \Gamma_\mf{m}(I)$ is acyclic, and since $\mf{m}$ was an arbitrary maximal ideal we conclude that the product $\prod_i F_i$ is N-totally acyclic.
\end{proof}

\section{Invariants of singularities}\label{section:invariants}

In this section we prove that the construction of $\kprofl{\tac}{X}$ only depends, in a sense we will make precise, on a certain subset of the singular locus of $X$. The idea, originally developed by Krause \cite[\S 6]{Krause05} for the homotopy category of injective sheaves, is to use local cohomology in the guise of a (co)localisation sequence, in our case (\ref{eq:local_cohom_recoll}) below. To be clear, in this section we return to the standing assumption that $X$ is a semi-separated noetherian scheme; no assumption is made about the existence of a dualising complex. 

\begin{setup} Fix an open subset $U \subseteq X$ with inclusion $f: U \lto X$, and set $Z = X \setminus U$.
\end{setup}

First we discuss a colocalisation sequence from \cite{Murfet07}. Restriction of sheaves from $X$ to $U$ sends complexes of flat sheaves to complexes of flat sheaves, and preserves pure acyclicity of such complexes, so there is a functor $(-)|_U: \kprof{X} \lto \kprof{U}$; see \cite[Definition 3.9]{Murfet07}. We will need some notation for the kernel of this functor.

\begin{definition} Let $\kprofl{Z}{X}$ denote the kernel of $(-)|_U: \kprof{X} \lto \kprof{U}$. A complex of flat sheaves $F$ belongs to this kernel if and only if $F|_U$ is pure acyclic. Since $(-)|_U$ is coproduct preserving, $\kprofl{Z}{X}$ is a localising subcategory of $\kprof{X}$.
\end{definition}

The restriction functor $(-)|_U$ and the inclusion $\kprofl{Z}{X} \lto \kprof{X}$ both have right adjoints, which we denote respectively by $\ndev f_*$ and $\ndev \Gamma_Z$, and the pair $(\ndev \Gamma_Z, \ndev f_*)$ is a colocalisation sequence \cite[Theorem 9.3]{Murfet07}
\begin{equation}\label{eq:local_cohom_recoll}
\xymatrix@C+2pc{
\kprof{U} \ar@<0.6ex>[r]^{\ndev f_*} & \kprof{X} \ar@<0.6ex>[l]^{(-)|_U} \ar@<0.6ex>[r]^{\ndev \Gamma_Z} & \kprofl{Z}{X} \ar@<0.6ex>[l]^{\inc}.
}
\end{equation}
Let us recall, roughly, what this means, and refer the reader to \cite[Definition 3.1]{Krause05} for the details. To say that a pair $(G: \cat{T} \lto \cat{T}'', F: \cat{T}' \lto \cat{T})$ \emph{is a colocalisation sequence} means that $F$ is fully faithful, that (up to equivalence) $G$ is the Verdier quotient of $\cat{T}$ by the image of $F$, and that $F$ has a left adjoint (equivalently, $G$ has a left adjoint). In this situation $G,F$ and the two left adjoints are arranged in a diagram of the form (\ref{eq:local_cohom_recoll}). In fact, (\ref{eq:local_cohom_recoll}) is even a recollement, but this fact will not be needed here.

There is an analogous recollement for the derived category, which relates $\qder{U}, \qder{X}$, and the derived category $\qderl{Z}{X}$ of complexes with cohomology supported on $Z$. The right adjoint of the inclusion $\qderl{Z}{X} \lto \qder{X}$ is the local cohomology functor $\rdev \Gamma_Z$ of Grothendieck, so (\ref{eq:local_cohom_recoll}) serves as an extension of local cohomology to $\kprof{X}$. The point of this section is that in certain sequences derived from (\ref{eq:local_cohom_recoll}), the third term (which is either a subcategory, or quotient of a subcategory, of $\kprofl{Z}{X}$) vanishes, and we deduce a relationship between an invariant of $U$ and $X$.

\begin{remark} When $U$ is affine, $\ndev f_*: \kprof{U} \lto \kprof{X}$ is just the usual direct image $f_*$ evaluated on a complex of flat sheaves, because the latter functor extends to the pure derived category of flat sheaves in this case \cite[Definition 3.9]{Murfet07}.
\end{remark}


\begin{lemma}\label{lemma:direct_image_tac} The functors $\ndev f_*$ and $\ndev \Gamma_Z$ preserve both acyclicity and N-total acyclicity of complexes of flat sheaves.
\end{lemma}
\begin{proof}
Let $F$ be an N-totally acyclic complex on $U$. Using \v{C}ech complexes we may write $F$ as an extension of direct sums of complexes of the form $h_*(F|_W)$, where $h: W \lto U$ is the inclusion of an affine open subset \cite[Proposition 3.13]{Murfet07}. To prove that $\ndev f_* F$ is N-totally acyclic, it therefore suffices to prove that each $\ndev f_*( h_* F|_W)$ is N-totally acyclic. Since $h$ is affine $\ndev h_* = h_*$, and the composite of right adjoints
\[
\kprof{W} \xlto{h_*} \kprof{U} \xlto{\ndev f_*} \kprof{X}
\]
is right adjoint to the composite $(-)|_W = (-)|^U_W \circ (-)|_U$. By the uniqueness of adjoints we have a natural equivalence $\ndev f_* \circ h_* \cong \ndev (fh)_*$, and the latter functor is just $(fh)_*$ since $fh$ is again affine. We already proved in Lemma \ref{lemma:inclusion_affine} that $(fh)_*$ preserves N-total acyclicity, whence the complex $\ndev f_*( h_* F|_W ) \cong (fh)_*(F|_W)$ is N-totally acyclic. 

Now, let $F$ be an N-totally acyclic complex of flat sheaves on $X$. Part of the data of the colocalisation sequence (\ref{eq:local_cohom_recoll}) is a triangle $\ndev \Gamma_Z(F) \lto F \lto \ndev f_*(F|_U) \lto \Sigma \ndev\Gamma_Z(F)$ in $\kprof{X}$. Both $F$ and $\ndev f_*(F|_U)$ are N-totally acyclic, so we infer from this triangle that $\ndev \Gamma_Z(F)$ is N-totally acyclic. Similarly, $\ndev f_*$ and $\ndev \Gamma_Z$ preserve acyclicity.
\end{proof}

\begin{definition} We write
\begin{align*}
\kprofl{Z,\ac}{X} := \kprofl{\ac}{X} \cap \kprofl{Z}{X},\\
\kprofl{Z,\tac}{X} := \kprofl{\tac}{X} \cap \kprofl{Z}{X}
\end{align*}
for the full subcategories of $\kprof{X}$ whose objects are, respectively, the acyclic and N-totally acyclic complexes that are ``supported'' on $Z$, in the sense that they restrict to pure acyclic complexes on $U$. Both are localising subcategories of $\kprof{X}$.
\end{definition}

From Lemma \ref{lemma:direct_image_tac}, we learn that $\ndev f_*$ and $\ndev \Gamma_Z$ restrict to functors
\begin{align*}
\ndev f_*|_{\tac} &: \kprofl{\tac}{U} \lto \kprofl{\tac}{X},\\
\ndev \Gamma_Z|_{\tac}&: \kprofl{\tac}{X} \lto \kprofl{Z,\tac}{X}.
\end{align*}
Similarly, there are restricted functors $\ndev f_*|_{\ac}$ and $\ndev \Gamma_Z|_{\ac}$. It is shown in \cite[Proposition 9.5]{Murfet07} that the pair $(\ndev \Gamma_Z|_{\ac}, \ndev f_*|_{\ac})$ is a colocalisation sequence, and we want to record here the analogue for N-total acyclicity.

\begin{proposition}\label{prop:recollement_tac_support} The pair $(\ndev \Gamma_Z|_{\tac}, \ndev f_*|_{\tac})$ is a colocalisation sequence
\[
\xymatrix@C+2pc{
\kprofl{\tac}{U} \ar@<0.6ex>[r]^{\ndev f_*|_{\tac}} & \kprofl{\tac}{X} \ar@<0.6ex>[l]^{(-)|_U} \ar@<0.6ex>[r]^{\ndev \Gamma_Z|_{\tac}} & \kprofl{Z,\tac}{X} \ar@<0.6ex>[l]^{\inc}.
}
\]
\end{proposition}
\begin{proof}
The functor $(-)|_U$ and its right adjoint $\ndev f_*$ restrict to an adjoint pair of functors between $\kprofl{\tac}{U}$ and $\kprofl{\tac}{X}$, and similarly $\ndev \Gamma_Z|_{\tac}$ remains the right adjoint of the inclusion $\kprofl{Z,\tac}{X} \lto \kprofl{\tac}{X}$. It is proven in \cite[Theorem 9.3]{Murfet07} that $\ndev f_*$ is fully faithful, so the restriction to $\kprofl{\tac}{U}$ is fully faithful. Finally, $\ndev \Gamma_Z|_{\tac}$ is a functor with a fully faithful left adjoint, and any such functor is (up to equivalence) the Verdier quotient of its kernel, which is $\kprofl{\tac}{U}$.
\end{proof}

Both $\ndev f_*|_{\ac}$ and $\ndev \Gamma_Z|_{\ac}$ preserve N-total acyclicity, so there are induced functors
\begin{align*}
\overline{\ndev} f_*&: \kprofl{\ac}{U}/\kprofl{\tac}{U} \lto \kprofl{\ac}{X}/\kprofl{\tac}{X},\\
\overline{\ndev} \Gamma_Z &: \kprofl{\ac}{X}/\kprofl{\tac}{X} \lto \kprofl{Z,\ac}{X}/\kprofl{Z,\tac}{X}.
\end{align*}
Allowing ourselves to abbreviate $\bold{N}_{\star}(\Flat -)$ to $\bold{N}_{\star}(-)$ for readability, we have:

\begin{proposition}\label{prop:recollement_actac_support} The pair $(\overline{\ndev} \Gamma_Z, \overline{\ndev} f_*)$ is a colocalisation sequence
\[
\xymatrix@C+2pc{
\bold{N}_{\ac}(U)/\bold{N}_{\tac}(U) \ar@<0.6ex>[r]^{\overline{\ndev} f_*} & \bold{N}_{\ac}(X)/\bold{N}_{\tac}(X) \ar@<0.6ex>[l]^{(-)|_U} \ar@<0.6ex>[r]^{\overline{\ndev} \Gamma_Z} & \bold{N}_{Z,\ac}(X)/\bold{N}_{Z,\tac}(X) \ar@<0.6ex>[l]^{\inc}.
}
\]
\end{proposition}
\begin{proof}
The functors $(-)|_U$ and $\ndev f_*|_{\ac}$ pass to the quotients and form an adjoint pair between $\bold{N}_{\ac}(U)/\bold{N}_{\tac}(U)$ and $\bold{N}_{\ac}(X)/\bold{N}_{\tac}(X)$. Moreover, because the unit and counit morphisms remain the same, and the counit is an isomorphism in $\kprofl{\ac}{U}$ by \cite[Theorem 9.3]{Murfet07}, we conclude that the same is true on the level of the quotients; in particular, the functor $\overline{\ndev} f_*$ is fully faithful. Similarly, $\overline{\ndev} \Gamma_Z$ is right adjoint to the fully faithful functor
\[
\bold{N}_{Z,\ac}(X)/\bold{N}_{Z,\tac}(X) \lto \bold{N}_{\ac}(X)/\bold{N}_{\tac}(X)
\]
induced by the inclusion $\bold{N}_{Z,ac}(X) \lto \bold{N}_{\ac}(X)$. It follows that $\overline{\ndev} \Gamma_Z$ is (up to equivalence) the Verdier quotient of $\bold{N}_{\ac}(X)/\bold{N}_{\tac}(X)$ by its kernel, which is clearly the image of $\overline{\ndev} f_*$, whence the pair $(\overline{\ndev} \Gamma_Z, \overline{\ndev} f_*)$ is a colocalisation sequence.
\end{proof}

\begin{remark}\label{remark:on_coloc_seq} In a colocalisation sequence, each row expresses the third category (reading in the direction of the arrows) as the quotient of the second by the first; for example, from Proposition \ref{prop:recollement_tac_support} we deduce that restriction $(-)|_U$ induces an equivalence
\begin{equation}\label{eq:left_hand_fraction}
\kprofl{\tac}{X}/\kprofl{Z,\tac}{X} \xlto{\sim} \kprofl{\tac}{U}.
\end{equation}
A similar observation applies in the context of Proposition \ref{prop:recollement_actac_support}.
\end{remark}

Next we examine conditions under which the fraction in (\ref{eq:left_hand_fraction}) has a vanishing denominator. We say that a local noetherian ring $A$ is a \emph{symmetric singularity} if it satisfies the equivalent conditions $(i)-(iv)$ of the following lemma.

\begin{lemma} Given a local noetherian ring $A$, the following are equivalent:
\begin{itemize}
\item[(i)] There exists a totally acyclic complex of projective $A$-modules which is not contractible. That is to say, $\kprojl{\tac}{A} \neq 0$.
\item[(ii)] There exists a non-projective Gorenstein projective $A$-module.
\item[(iii)] There exists an N-totally acyclic complex of flat $A$-modules which is not pure acyclic. That is to say, $\kprofl{\tac}{A} \neq 0$.
\item[(iv)] There exists a non-flat Gorenstein flat $A$-module.
\end{itemize}
If $A$ admits a dualising complex $D$, then the above are equivalent to:
\begin{itemize}
\item[(v)] The inclusion $\Thick(A, D) \subseteq \qderu{b}{\modd A}$ is proper.
\end{itemize}
\end{lemma}
\begin{proof}
The equivalence $(i) \Leftrightarrow (iii)$ follows from Lemma \ref{lemma:equiv_ktac_ntac}. For $(i) \Leftrightarrow (ii)$ we simply observe that an acyclic complex of projectives is contractible if and only if all its syzygies are projective, while $(iii) \Leftrightarrow (iv)$ follows from the fact that an acyclic complex of flats is pure acyclic if and only if all its syzygies are flat \cite[Theorem 8.6]{Neeman08}. The equivalence of $(v)$ with the other conditions in the presence of a dualising complex follows from \cite[Theorem 5.3]{Krause06}.
\end{proof}

\begin{definition} We say that a point $x \in X$ is a \emph{symmetric singularity} of $X$ if the local ring $\cat{O}_{X,x}$ is a symmetric singularity, in the above sense. 
\end{definition}

\begin{remark} Let $A$ be a noetherian local ring. There is a sequence of inclusions
\[
0 \subseteq \kprojl{\tac}{A} \subseteq \kprojl{\ac}{A},
\]
and it follows from the work of Iyengar-Krause \cite[Theorem 5.3]{Krause06} that $\kprojl{\ac}{A} = 0$ if and only if $A$ is regular, so a symmetric singularity is not regular. Any non-regular Gorenstein local ring $A$ is a symmetric singularity; in fact, for such $A$ we have
\[
0 \neq \kprojl{\ac}{A} = \kprojl{\tac}{A}.
\]
The inequality arises because $A$ is not regular, and the equality because $A$ is Gorenstein; see \cite[Corollary 5.5]{Krause06} and Corollary \ref{corollary:gorenstein_tac_ac_local}. 
\end{remark}


\begin{example}\label{example:non_gorenstein_symmetric} Let $S$ be a non-Gorenstein local ring. Then $R = S[[X,Y]]/(XY)$ is a non-Gorenstein local ring, and the $R$-module $R/YR$ is Gorenstein projective and not free, cf. \cite[Example 4.1.5]{Christensen00}, so $R$ is a symmetric singularity.
\end{example}

In the introduction we defined $\Perf(X)$ (resp. $\Coperf(X)$) to be the full subcategory of objects of finite flat dimension (resp. finite injective dimension) in $\qderu{b}{\coh X}$. We also defined
\begin{equation*}
\begin{split}
\SSg(X) &:= \qderu{b}{\cmodn{X}}/\Thick(\Perf(X), \Coperf(X)),\\
\NSg(X) &:= \Thick(\Perf(X), \Coperf(X))/\Perf(X).
\end{split}
\end{equation*}
For an open subset $U \subseteq X$, the restriction functor $(-)|_U: \qderu{b}{\coh X} \lto \qderu{b}{\coh U}$ sends $\Perf(X)$ into $\Perf(U)$, and $\Coperf(X)$ into $\Coperf(U)$, so we obtain canonical functors
\begin{gather*}
(-)|_U: \SSg(X) \lto \SSg(U), \text{ and } (-)|_U: \NSg(X) \lto \NSg(U).
\end{gather*}

\begin{proposition}\label{prop:invariance_ntac} If $U$ contains every symmetric singularity of $X$ then the functor
\begin{equation}\label{eq:invariance_ntac}
(-)|_U: \kprofl{\tac}{X} \lto \kprofl{\tac}{U}
\end{equation}
is an equivalence. When $X$ admits a dualising complex, it follows that the functor 
\begin{equation}\label{eq:invariance_ntac_2}
(-)|_U: \SSg(X) \lto \SSg(U)
\end{equation}
is an equivalence.
\end{proposition}
\begin{proof}
To say that $U$ contains every symmetric singularity means that $\kprofl{\tac}{\cat{O}_{X,z}}$ vanishes for every $z \in Z$. It follows that $\kprofl{Z,\tac}{X}$ is the zero category, and from the colocalisation sequence of Proposition \ref{prop:recollement_tac_support} we deduce the desired equivalence (\ref{eq:invariance_ntac}), via the observation of Remark \ref{remark:on_coloc_seq}. This restricts to an equivalence on compact objects and, when $X$ has a dualising complex, we infer from Proposition \ref{prop:compacts_ktac} that (\ref{eq:invariance_ntac_2}) is an equivalence up to direct summands. To see that (\ref{eq:invariance_ntac_2}) is essentially surjective, one uses a standard technique to extend a bounded complex of coherent sheaves on $U$ to $X$.
\end{proof}

\begin{proposition}\label{prop:invariance_nsg} If $U$ contains every singularity of $X$ that is not Gorenstein, then
\begin{equation}\label{eq:invariange_nsg}
(-)|_U: \kprofl{\ac}{X}/\kprofl{\tac}{X} \lto \kprofl{\ac}{U}/\kprofl{\tac}{U}
\end{equation}
is an equivalence. When $X$ admits a dualising complex, it follows that the functor
\begin{equation}\label{eq:invariange_nsg_2}
(-)|_U: \NSg(X) \lto \NSg(U)
\end{equation}
is an equivalence up to direct summands.
\end{proposition}
\begin{proof}
Under the stated conditions, the local ring at every point $z \in Z$ is Gorenstein, and thus $\kprofl{\tac}{\cat{O}_{X,z}} = \kprofl{\ac}{\cat{O}_{X,z}}$ by Theorem \ref{theorem:gorenstein_iff_actac}. We claim that the category
\[
\cat{S} := \kprofl{Z,\ac}{X}/\kprofl{Z,\tac}{X}
\]
vanishes under these conditions. By definition any complex $F$ in $\cat{S}$ restricts to a pure acyclic complex on $U$, so $F_x$ is N-totally acyclic for every $x \in U$. For every $z \in Z$ the complex $F_z$ is acyclic and therefore N-totally acyclic, as $\cat{O}_{X,z}$ is Gorenstein. Since $F$ is N-totally acyclic at every point, it is globally N-totally acyclic, and therefore zero in $\cat{S}$. This proves that $\cat{S} = 0$.

We may now use the colocalisation sequence of Proposition \ref{prop:recollement_actac_support} to argue that (\ref{eq:invariange_nsg}) is an equivalence. From Proposition \ref{prop:compacts_dgac} we infer that when $X$ has a dualising complex, (\ref{eq:invariange_nsg_2}) is an equivalence up to direct summands.
\end{proof}



\section*{Acknowledgements}

This project was initiated in June 2007, when Shokrollah Salarian was visiting the Australian National University. The authors thank the ANU, and in particular Professor Neeman, for providing a stimulating research environment.

\bibliographystyle{amsalpha}

\end{document}